\documentclass[11pt]{amsart}

\usepackage[utf8]{inputenc}
\usepackage[a4paper,twoside]{geometry}

\usepackage[usenames,dvipsnames]{xcolor}

\usepackage{enumitem}
\usepackage{amsmath, amssymb, amsthm}

\usepackage[english]{babel}
\usepackage{fancyhdr}
\usepackage[all]{xy}

\usepackage{tikz}
\usetikzlibrary{calc,decorations.markings,patterns}

\usepackage[framemethod=tikz]{mdframed}

\usepackage{float}
\usepackage{multirow}
\usepackage[colorlinks=true,linkcolor=blue,citecolor=red,backref=page,pdftex]{hyperref}

\newcommand{\retrait}{\hspace{1.7cm}}

\newcommand{\unp}{\mathbf{\mathrm{1 \kern-0.25em I}}}
\newcommand{\un}{\mathbf{1}}

\newcommand{\X}{\mathbf X}
\newcommand{\Hb}{\mathbf H}

\newcommand{\R}{\mathbb R}
\newcommand{\N}{\mathbb N}

\newcommand{\Q}{\mathbb Q}
\newcommand{\G}{\mathbf G}
\newcommand{\Z}{\mathbb Z}

\newcommand{\T}{\mathbb T}

\newcommand{\di}{{\rm d}}

\newcommand{\cal}{\mathcal}

\newenvironment{miniabstract}{%
\begin{center}\begin{minipage}{0.8\linewidth} %
}{%
\end{minipage} \end{center}\medskip %
}

\newcommand{\esp}{\mathbb E}
\newcommand{\prob}{\mathbb P}

\newcommand{\eupp}{\mathcal E_u^p}

\newcommand{\lib}{[} 
\newcommand{\rib}{]} 

\newcommand{\eup}{\mathcal E_u^1}

\newcommand{\supp}{\mathrm{supp}\,}

\newtheorem{theorem}{Theorem}[section]
\newtheorem*{theorem*}{Theorem}
\newtheorem*{lemma*}{Lemma}
\newtheorem*{corollary*}{Corollary}
\newtheorem{lemma}[theorem]{Lemma}
\newtheorem{proposition}[theorem]{Proposition}
\newtheorem*{proposition*}{Proposition}
\newtheorem{corollary}[theorem]{Corollary}

\theoremstyle{definition}
\newtheorem{definition}[theorem]{Definition}
\newtheorem*{definition*}{Definition}
\newtheorem{example}[theorem]{Example}

\newtheorem*{notations*}{Notations}
\theoremstyle{remark}
\newtheorem{remark}[theorem]{Remark}

\numberwithin{equation}{section}
\begin{document}
\pagestyle{plain}
\thispagestyle{plain}

\title{Central limit theorem and law of the iterated logarithm for the linear random walk on the torus}


\author{Jean-Baptiste Boyer}
\email{jeaboyer@math.cnrs.fr}



\date{\today}


\begin{abstract}
Let $\rho$ be a probability measure on $\mathrm{SL}_d(\mathbb{Z})$ and consider the random walk defined by $\rho$ on the torus $\mathbb{T}^d=  \mathbb{R}^d/\mathbb{Z}^d$.

Bourgain, Furmann, Lindenstrauss and Mozes proved that under an assumption on the group generated by the support of $\rho$, the random walk starting at any irrational point equidistributes in the torus. 

In this article, we study the central limit theorem and the law of the iterated logarithm for this walk starting at some point having good diophantine properties.
\end{abstract}

\maketitle

\tableofcontents

\section{Introduction}

Let $\Gamma$ be a subgroup of $\mathrm{SL}_d(\Z)$ and $\rho$ a probablity measure on $\Gamma$. The action of $\Gamma$ on the torus $\X:=\T^d = \R^d/\Z^d$ allows one to define a random walk, setting, for any $x\in \X$,
\[
\left\{\begin{array}{rl}
X_0 &=x \\ X_{n+1} &= g_{n+1} X_n
\end{array} \right.
\]
where $(g_n)\in\Gamma^\N$ is chosen with the law $\rho^{\otimes \N}$. We note $\prob_x$ the measure on $\X^\N$ associated to the random walk starting at $x$.

The Markov operator associated to the walk is the one defined for any non-negative borelian function $f$ on $\X$ and any $x\in \X$ by
\[
Pf(x)  = \int_\G f(gx) \di\rho(g)
\]
We note $\nu$ the Lebesgue measure on $\X$. As $\nu$ is $\Gamma-$invariant, it is also $P-$invariant : for any continuous function $f$ on $\X$,
\[
\int_\X Pf \di \nu = \int_\X f\di\nu
\]
One can prove that for any $p\in[1,+\infty]$, $P$ is a continuous operator on $\mathrm{L}^p(\T^d,\nu)$ and $\|P\|_{\mathrm{L}^p} = 1$.

\medskip
In the sequel, we will need an hypothesis telling that the support of $\rho$ is big.

Let $\Hb$ be a closed subgroup of $\mathrm{SL}_d(\R)$.
We say that the action of $\Hb$ on $\R^d$ is \emph{strongly irreducible} if $\Hb$ doesn't fix any finite union of proper subspaces of $\R^d$ and we say that the action is \emph{proximal} if there is some $h\in\Hb$ for which there are an $h-$invariant line $V_h^+$ in $\R^d$ and an $h-$invariant hyperplane $V_h^<$ such that $\R^d= V_h^+ \oplus V_h^<$ and the restriction of $h$ to $V_h^<$ has a spectral radius strictly smaller than the restriction of $h$ to $V_h^+$.

We say that a borelian probability measure $\rho$ on $\G=\mathrm{SL}_d(\R)$ has an exponential moment if for some $\varepsilon \in \R_+^\ast$ we have
\[
\int_\G \|g\|^\varepsilon \di\rho(g) <+\infty
\]
Under these assumptions (exponential moment and strongly irreducible and proximal action of the closed subgroup generated by the support of $\rho$), we know that $P$ has a spectral radius strictly smaller than $1$ in the orthogonal of the constant functions in $\mathrm{L}^2(\X,\nu)$ (cf~Furmann and Shalom in~\cite{FuSh99} and also Guivarc'h in~\cite{Gu06}). We will say in that case that $P$ has a  \emph{spectral gap} in $\mathrm{L}^2(\X,\nu)$.

In particular under these assumptions, for any function $f\in \mathrm{L}^2(\X,\nu)$, there is a function $g\in\mathrm{L}^2(\X,\nu)$ such that $f=g-Pg + \int f\di\nu$ and the law of large numbers and the central limit theorem are already known for $\nu-$a.e. starting point $x\in \T^d$ (see for instance~\cite{GL94}, \cite{BIS95} and~\cite{DL03}) the variance in the central limit theorem beeing 
\begin{equation} \label{equation:variance}
\sigma^2(f) = \int g^2 - (Pg)^2 \di\nu
\end{equation}

\medskip
In this article, we are interested in the study of the walk starting at an arbitrary point $x\in \T^d$.

It is easy to see that the rational points in $\T^d$ have a finite $\Gamma-$orbit since any $g\in \Gamma$ increases the denominator of such a point. So to study the walk starting at a rational point one can use the classical results for Markov chains with a finie number of states.

We define a measurable application $\nu : \X \to \cal M^1(\X)$ (the set of probability measures on $\X$) by $\nu_x =\nu$ (Lebesgue measure on $\X$) if $x\not\in \Q^d/\Z^d$ and $\nu_x$ is the equidistributed measure on $\Gamma_\rho x$ if $x\in \Q^d/\Z^d$ where $\Gamma_\rho$ is the subgroup of $\mathrm{SL}_d(\Z)$ generated by the support of $\rho$.

\medskip
Bourgain, Furmann, Lindenstrauss and Mozes proved the following
\begin{theorem}[\cite{BFLM11}]
Let $\rho$ be a probability measure on $\Gamma= \mathrm{SL}_d(\Z)$ having an exponential moment and whose support generates a strongly irreducible and proximal subgroup.

Then, for any $x\in \X$, any continuous function $f$ on $\X$ and $\rho^{\otimes \N}-$a.e. $(g_n)\in \Gamma^\N$,
\[
\frac 1 n \sum_{k=0}^{n-1}f(g_k \dots g_1 x) \xrightarrow\, \int f\di\nu_x
\]
\end{theorem}

This theorem is the law of large numbers for the sequence $(f(g_n \dots g_1 x))_{n\in \N}$ and we would like to study the central limit theorem and the law of the iterated logarithm. First, we will look at conditions on the function $f$ for which the variance given by equation~\ref{equation:variance} vanishes. This was studied in~\cite{FuSh99} when the measure $\rho$ is aperiodic : it's support is not contained  in a class modulo a proper subgroup of $\G$.

Let $\G$ be a locally compact group acting continuously on a topological space $\X$ preserving the probability measure $\nu$.

We say that the action of $\G$ on $\X$ is $\nu-$ergodic if every measurable $\G-$invariant function is constant $\nu-$a.e.

We will prove next
\begin{proposition*}[\ref{proposition:nullite_variance}]
Let $\G$ be a locally compact group acting continuously and ergodically on a topological space $\X$ endowed with a $\G-$invariant probability measure $\nu$.

Then, for any $g\in \mathrm{L}^2(\X,\nu)$, the following assertions are equivalent
\begin{enumerate}
\item $\|Pg\|_2 = \|g\|_2$
\item There is some subgroup $\Hb$ of $\G$ and some $\gamma\in \G$ such that $g$ is $\Hb-$invariant and $\supp \rho \subset \Hb \gamma$.
\end{enumerate}
\end{proposition*}

\begin{remark}
In particular, there is a non-constant function $g\in \mathrm{L}^2(\X,\nu)$ such that $\|Pg\|_2 = \|g\|_2$ if and only if there is a subgroup $\Hb$ of $\G$ whose action on $\X$ is not ergodic and some element $\gamma \in \G$ such that $\supp\rho \subset \Hb \gamma$.
\end{remark}

From now on, we fix a norm $\|\,.\,\|$ on $\R^d$ which defines a distance on the torus, setting, for any $x,y\in \X$,
\[
d(x,y) = \inf_{p\in \Z^d} \| \overline{x}-\overline{y}-p\|
\]
Where $\overline{x}$ (resp. $\overline{y}$) is a representative of $x$ (resp. $y$) in $\R^d$.

We note $C^{0,\gamma}(\X)$ the space of $\gamma-$hölder continuous functions on $\X$ that we endow with the norm : for any $f\in \cal C^{0,\gamma}(\X)$,
\begin{equation} \label{equation:norme_gamma}
\|f\|_\gamma = \sup_{x\in \X} |f(x)| + \sup_{\substack{x,y\in \X\\ x\not=y}}\frac{|f(x)-f(y)|}{d(x,y)^\gamma}
\end{equation}

\medskip
To prove the central limit theorem and the law of the iterated logarithm, we will see in section~\ref{section:tore_point_diophantien} that the result of Bourgain, Furmann, Lindenstrauss and Mozes in~\cite{BFLM11} allows one to have a speed of convergence depending on the diophantine properties of $x$ (when $f$ is hölder continuous). This will give us the

\begin{theorem*}[\ref{theorem:TCL_tore_diophantien}]
Let $\rho$ be a probability measure on $\Gamma= \mathrm{SL}_d(\Z)$ having an exponential moment and whose support generates a strongly irreducible and proximal subgroup.

Then, for any $\gamma \in ]0,1]$ there is $\beta_0 \in \R_+^\ast$ such that for any $B \in \R_+^\ast$ and $\beta \in ]0,\beta_0[$ we have that for any irrational point $x\in \X$ such that the inequality
\[
d\left(x,\frac p q \right) \leqslant e^{-Bq^\beta}
\]
has a finite number of solutions $p/q \in \Q^d/\Z^d $, we have that for any $\gamma-$holder continuous function $f$ on the torus, noting $\sigma^2(f) $ the variance given by equation~\ref{equation:variance} we have that
\[
\frac 1 {\sqrt n} \sum_{k=0}^{n-1} f(X_k) \xrightarrow{\mathcal{L}} \cal N\left(\int f\di \nu, \sigma^2(f) \right)
\]
(If $\sigma^2=0$, the law $\cal N(\mu,\sigma^2)$ is a Dirac mass at $\mu$).

Moreover, if $\sigma^2(f) \not=0$ then, $\prob_x-$a.e.,
\[
\liminf \frac{ \sum_{k=0}^{n-1} f(X_k) - \int f \di \nu}{\sqrt{2n\sigma^2(f) \ln\ln n}} =-1 \text{ and }\limsup \frac{ \sum_{k=0}^{n-1} f(X_k) - \int f \di \nu}{\sqrt{2n\sigma^2(f) \ln\ln n}} =1 
\]
and if $\sigma^2(f)=0$, then for $\nu-$a.e. $x\in \X$, the sequence $ (\sum_{k=0}^{n-1} f(X_k) - \int f \di \nu)_n$ is bounded in $\mathrm{L}^2(\prob_x)$.
\end{theorem*}

\begin{remark}
Our condition is satisfied in particular for diophantine points of the torus. Therefore, the set of points where our theorem applies has Lebesgue-measure $1$. But, the theorem also works for some Liouville numbers.
\end{remark}

Our strategy to prove this result is \emph{Gordin's method}. For a continuous function on the torus, we call \emph{Poisson's equation} the equation $f=g-Pg+\int f\di\nu$ where $g$ is some unknown function. If this equation has a continuous solution (we already know, with a spectral gap argument, that a solution exists in $\mathrm{L}^2(\X,\nu)$) then, we can write for any $x\in \X$ and $\rho^{\otimes \N}-$a.e. $(g_n) \in \Gamma^\N$,
\[
\sum_{k=0}^{n-1} f(X_k) =g(X_0) - g(X_n) + \sum_{k=0}^{n-1} g(X_{k+1}) - Pg(X_k)
\]
The key remark is that $M_n = \sum_{k=0}^{n-1} g(X_{k+1}) - Pg(X_k)$ is a martingale with bounded increments and we can use the classic results for the martingales to prove the central limit theorem and the law of the iterated logarithm.

\medskip
Here, in general, there cannot exist a continuous function $g$ on the torus such that $f=g-Pg+\int f\di\nu$ because this would imply that $f$ has the same integral against all the stationary measures (in particular, we would have that $f(0) = \int f\di\nu$). However, we will prove the theorem by showing that for any holder continuous function $f$ on the torus, we can solve Poisson's equation at points having good diophantine properties (cf. section~\ref{section:tore_point_diophantien}). Moreover, the solution we construct will not be bounded on $\X$ but will be dominated by a function $u:\X\to [1,+\infty]$ that we call \emph{drift function} and that satisfies
\[
Pu \leqslant au+b
\]
for some $a\in ]0,1[$ and $b\in \R$. This equation means that if $u(x)$ is large, then, in average, $u(gx)$ is much smaller than $u(x)$. Or in other words, the function $g$ that we construct is not bounded but the walk doesn't spend much time at points $x$ where $|g(x)|$ is large.

\bigskip
The first section of this article consists in a study of  drift functions and the proof of the central limit theorem and the law of the iterated logarithm for martingales with difference sequence bounded by drift functions.

In the second section, we study the variance appearing in the central limit theorem and the case where it vanishes.

Finally, in the third section, we solve Poisson's equation for points of the torus having good diophantine properties and we prove theorem~\ref{theorem:TCL_tore_diophantien}.

\section{Drift functions} \label{section:induction}
\thispagestyle{empty}
\begin{miniabstract}
In this section, we introduce and study some kind of functions that we call ``drift functions'' and that allow one to control the sequence $(f(X_n))$ when $f$ is dominated by one.

Moreover, we prove the law of large number, the central limit theorem and the law of the iterated logarithm for martingales dominated by drift functions.
\end{miniabstract}

\subsection{Definitions} \label{section:drift}

In this section, $(X_n)$ is a Markov chain on a standard borelian space $\X$.

\begin{definition}[Drift function]\label{def:drift}Let $u: \X\to [1,+\infty]$ be a borelian function and $C$ a borelian subset of $\X$.

We say that $(u,C)$ is a drift function if $u$ is bounded on $C$ and if there is some $b\in \R$ such that
\[
Pu \leqslant u + b\un_C
\]
In general, we will say that $u$ is a drift function without indicating the set $C$.
\end{definition}

\begin{remark}
These functions are studied by many authors and our main reference is~\cite{MT93} (see also~\cite{GM96}).

Meyn and Tweedie don't assume that $u$ is bounded on $C$ but that $C$ is a so called \emph{petite-set} and this allows them to prove that one can find a borelian set $C'$ such that $(u,C')$ is a drift function with our definition.
\end{remark}

\begin{remark}
Many authors call Lyapunov function any non negative measurable function $v: \X\to [1,+\infty[$ such that $Pv\leqslant v$. So, our drift functions are very close to Lyapunov functions.
\end{remark}

As we assume that $Pu \leqslant u+b\un_C$, we can study borelian functions $f$ on $\X$ such that
\begin{equation}
|f| \leqslant u-Pu+b\un_C \label{equation:derive_f}
\end{equation}
We are going to see that we have a good control on the sequence $(P^nf)$ (or, more specifically, on the series whose general terms involves the $P^nf$).

Therefore, we set, for $p\in \R_+$,
\[
\cal E_u^p:= \left\{   f:\X\to \R\middle| f\text{ is borelian and }\exists M \forall x\in \X, \;|f(x)| \leqslant M(u-Pu+b\un_C)^{1/p} \right\}
\] 
And, for any $f\in \cal E_u^p$, we set
\[
\|f\|_{\cal E_u^p} = \inf\left\{ M\in \R \middle| \forall x\in \X,\;| f(x)| \leqslant M(u-Pu+b\un_C)^{1/p} \right\}
\]
\begin{remark}
The space $(\cal E_u^p,\|\,.\,\|_{\cal E_u^p})$ is a Banach space.
\end{remark}

In the same way, we set, for any $p\in [1,+\infty[$,
\[
\cal F_u^p:= \left\{   f:\X\to \R\middle| f\text{ is borelian and }\exists M, \forall x\in \X, \;|f(x)| \leqslant Mu(x)^{1/p} \right\}
\] 
and, for $f\in \cal F_u^p$,
\[
\|f\|_{\cal F_u^p} = \sup_{x\in \X} \frac{ |f(x)|}{u(x)^{1/p}}
\]

In next lemma, we use the control given by the drift function to prove that the space $\eup$ is a subset of the space of integrable functions against the stationary measures for the Markov chain.

\begin{lemma}\label{lemma:extmesures}
Let $u$ be a drift function and $\nu$ a borelian probability measure on $\X$ that is $P-$stationnary and such that $\nu(u<+\infty)=1$.

Then, the identity operator defined from $\eupp(\X)$ to $ \mathcal L^p(\X,\nu)$ is continuous.
\end{lemma}

\begin{proof}
{(cf. lemma~3.8 in~\cite{BQstat3})}~

\medskip
Let $f\in \eupp$ be a non negative function, $x\in \X$ and $n\in\N^\ast$, then, by definition of $\eupp$, $|f|^p(x)\leqslant \|f\|_{\eupp}^p (u-Pu+b)(x)$, and so,
\[\frac 1 n \sum_{k=0}^{n-1}P^k(|f|^p)(x)\leqslant \frac{\|f\|_{\eupp}^p} n(u-P^nu+nb)\leqslant\|f\|_{\eupp} ^p(\frac 1 n u(x)+b)\]
But, according to Chacon-Ornstein's ergodic theorem (see for instance the theorem 3.4 of the third chapter of~\cite{Kre85}), there is a $P-$invariant function $f^\ast$ on $\X$ that is non negative and such that $\int |f|^p\di\nu=\int f^\ast \di\nu$ and, for $\nu-$a.e. $x\in \X$,
\[\frac 1 n \sum_{k=0}^{n-1}P^k|f|^p(x) \xrightarrow\, f^\ast(x) \]
But, since $u$ is finite $\nu-$a.e., we get that $f^\ast(x)\leqslant b\|f\|_{\eupp}^p$ for $\nu-$a.e. $x\in \X$. And so, $f^\ast\in\mathrm L^\infty(\X,\nu)\subset \mathrm L^1(\X,\nu)$ since we assumed that $\nu$ is a probability measure. This proves that, $f\in\mathcal L^p(\X,\nu)$ and that $\|f\|_{\mathcal L^p(\X,\nu)}\leqslant b^{1/p}\|f\|_{\eupp}$.
\end{proof}

\subsection{The LLN, the CLT and the LIL for martingales} \label{section:LLN_TCL}

\begin{miniabstract}
In this section, we prove three of the classical results in probability theory for martingales with increments dominated by a drift function.

In particular we will prove that the central limit theorem and the law of the iterated logarithm for martingales can be deduced from a law of large numbers and this will be our corollary~\ref{corollary:TCL_martingales}.
\end{miniabstract}

\begin{remark}
In this section, we make an assumption such as  ``$f\in \cal E_u^p $ for some $p>1$'' very often. The reader shall not be afraid of this assumption because in many examples we can construct families of drift functions and if $f$ is dominated by one, $f^p$ will be dominated by an other one.
\end{remark}

Before we state and prove corollary~\ref{corollary:TCL_martingales}, we state some lemmas that we will also use in the study of the random walk on the torus.

\medskip
First, we extend the law of large numbers for martingales (stated in~\cite{Br60}) for measurable functions $f\in \cal E_u^p$ for some $p>1$: this will be our proposition~\ref{proposition:lln_martingales}. To prove it, we will use the following
\begin{lemma}\label{lemma:sommation_Abel}
Let $u$ be a drift function, $x\in \X$, and $\alpha\in\R_+$, then
\[ \sup_{n\in\N}\sum_{k=0}^{n} \frac {P^{k}(u-Pu)(x)}{(k+1)^{\alpha}} \leqslant u(x)\]
\end{lemma}

\begin{proof}
We can compute :
\begin{flalign*}
\sum_{k=0}^{n} \frac {P^{k}(u-Pu)}{(k+1)^{\alpha}} &= \sum_{k=0}^n \frac 1 {(k+1)^\alpha} P^ku -\sum_{k=0}^n \frac 1 {(k+1)^\alpha} P^{k+1} u \\
& =\sum_{k=1}^n (\frac 1 {(k+1)^\alpha}-\frac 1 {k^\alpha})P^ku + u-\frac 1 {(n+1)^\alpha} P^{n+1} u \\
& \leqslant u(x) \text{ since }u\text{ is non negative}
\end{flalign*}
\end{proof}

Then, we prove that it is the same thing to study
\[
\frac 1 n\sum_{k=0}^{n-1} f(X_n) - Pf(X_n)
\text{ and }
\frac 1 n \sum_{k=0}^{n-1} f(X_{n+1}) - Pf(X_n).
\]
Thus, having the law of large numbers, the central limit theorem and the law of the iterated logarithm for martingales, we will get these results for functions $f$ that writes $f=g-Pg$.

\begin{lemma}\label{lemma:u(X_n)/n}
Let $u$ be a drift function and $p>1$. Then, for any $f\in \eupp$, any $x\in \X$ such that $u(x)$ is finite and any $\varepsilon \in ]0,p[$,
\[
\frac {f(X_n)}{n^{1/(p-\varepsilon)}}  \xrightarrow\, 0 \; \prob_x-\text{a.e. and in }\mathrm{L}^1(\prob_x)
\]
\end{lemma}

\begin{remark}
We will use this lemma with $p>1$ and $p-\varepsilon=1$ and with $p>2$ and $p-\varepsilon=2$.
\end{remark}

\begin{proof}
With the notations of the lemma, let's compute, for any $n\in \N$,
\[
\esp_x |f(X_n)|^{p} \leqslant \|f\|_{\eupp} \esp_x u(X_n) - Pu(X_n) + b = \|f\|_{\eupp} P^n(u-Pu+b)
\]
And so, assuming without any loss of generality that $ \|f\|_{\eupp} = 1$, we get
\begin{flalign*}
\sum_{k=0}^n \frac{\esp_x |f(X_k)|^{p}}{(k+1)^{p/(p-\varepsilon)}} &\leqslant \sum_{k=0}^n \frac{P^k(u-Pu)}{(k+1)^{1 + \varepsilon/(p-\varepsilon)}} + b\sum_{k=0}^n \frac 1 {(k+1)^{1 + \varepsilon/(p-\varepsilon)}} \\& \leqslant u(x) + b \sum_{n\in \N^\ast} \frac 1 {n^{1 + \varepsilon/(p-\varepsilon)}} 
\end{flalign*}
where we used lemma~\ref{lemma:sommation_Abel} to control the first sum.

Thus, for any $x\in \X$ such that $u(x)$ is finie,
\[
\sum_{k=0}^{+\infty} \esp_x \left( \frac{|f(X_k)|}{(k+1)^{1/(p-\varepsilon)}} \right)^{p}
\]
is finite and this finishes the proof.
\end{proof}

\begin{proposition}
\label{proposition:lln_martingales}
Let $u$ be a drift function and $p\in ]1,+\infty[$. For any $f\in\eupp$ and $x\in X$,
\[\frac 1 n \sum_{k=0}^{n-1} f(X_{k+1})-Pf(X_k)\xrightarrow \, 0  \;\;\prob_x-a.e.\textnormal{ and in } \mathrm L^{p}(\prob_x)\]
\end{proposition}

\begin{proof}
For any $n\in \N^\ast$, let $M_n= \sum_{k=0}^{n-1} f(X_{k+1})-Pf(X_k)$.

Then, $(M_n)$ is a martingale with $\esp M_n = 0$ and
\begin{flalign*}
\esp_x |M_{n+1}-M_n|^{p} &=\esp_x |f(X_{n+1})-Pf(X_{n})|^{p}=P^n(\esp_x|f(X_1)-Pf(x)|^{p}) &\\
& \leqslant 2^{p-1} P^{n+1}(|f|^{p})(x)\leqslant 2^{p-1}\|f\|_{\cal E_u^p} P^{n+1}(u-Pu+b)& 
\end{flalign*}
Thus,
\begin{flalign*}
\sum_{n=1}^{+\infty} \frac 1 {n^{p}} \esp_x |M_{n+1}-M_n|^{p} &\leqslant 2^{p-1}\|f\|_{\cal E_u^p}\sum_{n=1}^{+\infty} \frac {P^{n+1}(u-Pu+b)}{n^{p}} \\
&\leqslant 2^{p-1}\|f\|_{\cal E_u^p} \left( u(x)+b\sum_{k=0}^{+\infty} \frac 1 {n^{p}}\right)
\end{flalign*}
And so, according to the law of large numbers for martingales (see the theorem 2.18 in \cite{HH80}), we get that $\frac 1 n M_n\xrightarrow \,0$ $\prob_x-$a.e. and in $\mathrm L^{p}(\prob_x)$.
\end{proof}

\begin{lemma} \label{lemma:produit_cesaro}
Let $u$ be a drift function such that $Pu \leqslant au+b$ for some $a\in ]0,1[$ and $b\in \R$ and let $p\in ]1,+\infty[$.

Let $\psi: \N \to \R_+$ be a decreasing function converging to $0$ at $+\infty$.

Then, for any $f\in \cal E_u^p$,
\[
\frac 1 n \sum_{k=0}^{n-1} \psi(k) f(X_k) \xrightarrow\,0 \;\;\prob_x-\text{a.e. and in }\mathrm{L}^p(\prob_x)
\]
\end{lemma}

\begin{proof}We shall assume without any loss of generality that $\|f\|_{\cal E_u^p}=1$.

To prove the convergence in $\mathrm{L}^p$, we compute
\begin{flalign*}
\left(\esp_x \left|\frac 1 n \sum_{k=0}^{n-1} \psi(k) f(X_k) \right|^p \right)^{1/p}&\leqslant \frac 1 n \sum_{k=0}^{n-1} \psi(k) \left(\esp_x|f(X_k)|^p\right)^{1/p}\\& \leqslant \frac 1 n \sum_{k=0}^{n-1} \varphi_k  \left( P^k u(x) \right)^{1/p} \\
& \leqslant  \frac 1 n \sum_{k=0}^{n-1} \varphi_k ( u(x) + b/(1-a))^{1/p}
\end{flalign*}
And we conclude with Cesaro's lemma.

\medskip
To study the a.e.-convergence, in a first time, we are going to prove the for any $x$ such that $u(x)$ is finite,
\begin{equation}\label{equation:lemme:produit_cesaro}
\limsup_{n} \frac 1 n \sum_{k=0}^{n-1} \left| f(X_k) \right| \leqslant\frac  {b(1+b)^{1/p}}  {1-a^{1/p}}\;\;\prob_x-\text{p.s.}
\end{equation}
First, we remark that for any $x\in \X$,
\[
|f(x)|^{p} \leqslant u(x) - Pu(x) + b \leqslant (1+b)u(x)
\]
Moreover, for any $r\in]0,1]$, we note $u_r$ the function defined by $u_r(x)=u(x)^r$. And so, using the concavity of the function $(t\mapsto t^r)$, we get that
\[
Pu_r \leqslant(Pu)^r \leqslant (au+b)^r \leqslant a^r u_r + b
\]
This means that $u_r \leqslant \frac 1{1-a^r} (u_r-Pu_r+b)$. And so, setting $r=1/p$, we obtain
\begin{flalign*}
\frac 1 n \sum_{k=0}^{n-1} \left| f(X_k) \right|& \leqslant \frac {(1+b)^{1/p}} n \sum_{k=0}^{n-1} u_{1/p}(X_k) \\
& \leqslant \frac  {(1+b)^{1/p}}  {1-a^{1/p}} \frac 1 n\sum_{k=0}^{n-1} u_{1/p}(X_k) - P u_{1/p}(X_k) + b \\
& \leqslant  \frac  {(1+b)^{1/p}}  {1-a^{1/p}} \left(\frac 1 n u(x) + b\right) + \frac  {(1+b)^{1/p}}  {1-a^{1/p}} \frac 1 n \sum_{k=0}^{n-1} u_{1/p}(X_{k+1}) - Pu_{1/p}(X_k)
\end{flalign*}
Moreover, by definition, $u_{1/p}^p = u \in \cal E_u^1$ and so, using proposition~\ref{proposition:lln_martingales}, we get
\[
\frac 1 n \sum_{k=0}^{n-1} u_{1/p}(X_{k+1}) - Pu_{1/p}(X_k)\xrightarrow\,0 \;\prob_x-\text{a.e.}
\]
This proves inequality~\ref{equation:lemme:produit_cesaro}. 

Thus, a.e., there is some $n_0\in \N$ such that for $n\geqslant n_0$,
\[
\frac 1 n \sum_{k=0}^{n-1} |f(X_k)| \leqslant 2 \frac  {b(1+b)^{1/p}}  {1-a^{1/p}}
\]
And so, for $n$ such that $\sqrt n\geqslant n_0$, we get
\begin{flalign*}
\frac 1 n \sum_{k=0}^{n-1} \psi(k) |f(X_k)| &\leqslant \frac {\psi(0)} n \sum_{k=0}^{\lfloor \sqrt{n} \rfloor -1} |f(X_k)| + \frac{\psi(\lfloor \sqrt n\rfloor)} n \sum_{k=\lfloor \sqrt n\rfloor}^{n-1} |f(X_k)| \\
& \leqslant  2 \frac  {b(1+b)^{1/p}}  {1-a^{1/p}} \left(\frac{\psi(0) }{\sqrt n} + \psi(\sqrt n) \right)  
\end{flalign*}
And, as $\psi$ converges to $0$, this finishes the proof of the lemma.
\end{proof}

Using the same ideas as in the proof of proposition~\ref{proposition:lln_martingales}, we can prove the
\begin{lemma}\label{lemma:lindeberg}
Let $u$ be a drift function and $p>2$.

Let $g\in\eupp$ and $x\in\X$ such that $u(x)$ is finite.

Then, for any $\varepsilon\in \R_+^\ast$
\[\frac 1 n \sum_{k=0}^{n-1} \esp_x \left((g(X_{k+1})-Pg(X_k))^2 \un_{ |g(X_{k+1})- Pg(X_k)|\geqslant\varepsilon\sqrt n}\right)\xrightarrow[n\to +\infty]\, 0\]
and
\[
\sum_{n=1}^{+\infty} \frac1 {\sqrt n} \esp_x \left(\left|g(X_{n+1})-Pg(X_n) \right| \un_{ |g(X_{n+1})- Pg(X_n)|\geqslant\varepsilon\sqrt n} \right)\text{ is finite}
\]
Finally, there is $\delta \in\R_+^\ast$ such that
\[
\sum_{n=1}^{+\infty} \frac 1 {n^2} \esp_x \left((g(X_{n+1})-Pg(X_n))^4 \un_{ |g(X_{n+1})- Pg(X_n)|\leqslant\delta \sqrt n}\right) 
\]
is finite.
\end{lemma}

\begin{proof}
Using Markov's inequality, we can compute
\[
\esp_x \left(h(X_{k+1}, X_k)^2\un_{ |h(X_{k+1}, X_k)|\geqslant\varepsilon\sqrt n}\right)
\leqslant\frac { P^k\left(\esp\left(g(X_1)-Pg(X_0)\right)^{p}\right)} {\varepsilon^{p-2} n^{(p-2)/2}}
\]
where we noted $h(x,y)= g(x) - Pg(y)$.

But, $\esp_x\left[\left((g(X_1)-Pg(X_0)\right)^{p}\right]\in\eup$, since we took $g$ in $\eupp$.

So,
\begin{flalign*}
\frac 1 n \sum_{k=1}^n \esp_x \left(h(X_{k+1}, X_k)^2\un_{ |h(X_{k+1},X_k)|\geqslant\varepsilon\sqrt n}\right)
&\leqslant \frac C {n^{1+(p-2)/2} \varepsilon^{p-2}} \sum_{k=0}^{n-1} {P^k(u-Pu+b) }\\
&\leqslant  \frac C {n^{1+(p-2)/2} \varepsilon^{p-2}} u(x) + \frac {bC} {n^{(p-2)/2} \varepsilon^{p-2}}
\end{flalign*}
And the right side converges to $0$ since $u(x)$ is finite.

\medskip
The two sums that we have to study are bounded by constants times
\[
\sum_{n=1}^{+\infty} \frac {1} {n^{1+(p-2)/2}} \esp_x \left(\left|g(X_{n+1})-Pg(X_n)\right|^{p} \right) 
\]
and, once again, using that $g\in \eupp$, we get that
\[
 \esp_x \left(\left|g(X_{n+1})-Pg(X_n)\right|^{p} \right)  \leqslant \|g\|_{\eupp}  P^n(u-Pu+b)
\]
And we shall conclude with lemma~\ref{lemma:sommation_Abel}.
\end{proof}

Lemma~\ref{lemma:lindeberg} is important since it is a first step in the proof of the central limit theorem and the law of large numbers as we will see in next
\begin{corollary} \label{corollary:TCL_martingales} Let $u$ be a drift function and $p>2$.

Let $g\in\eupp$ and $x\in\X$ such that $u(x)$ is finite.

If
\[
\frac 1 n \sum_{k=0}^{n-1} P(g^2)(X_k) - (Pg(X_k))^2
\]
converges in $\mathrm{L}^1(\prob_x)$ and a.e. to some constant $\sigma^2(g,x)$,
then,
\[
\frac 1 {\sqrt n} \sum_{k=0}^{n-1} g(X_{k+1})-Pg(X_k) \xrightarrow[n\to\infty]{\mathcal L}\cal N(0,\sigma^2(g,x))
\]
Where we noted $\cal N(0,0)$ the Dirac mass at $0$.

Moreover, if $\sigma^2(g,x) \not=0$ then,
\[
\limsup \frac{\sum_{k=0}^{n-1} g(X_{k+1}) - Pg(X_k)}{\sqrt{2n \sigma^2(g,x) \ln\ln(n)}} =1 \;a.e.
\]
and
\[
\liminf \frac{\sum_{k=0}^{n-1} g(X_{k+1}) - Pg(X_k)}{\sqrt{2 n \sigma^2(g,x) \ln\ln(n)}} =-1 \;a.e.
\]
\end{corollary}

\begin{proof}
The central limit theorem comes from Brown's one (cf~\cite{Bro71}) since the `` Lindeberg condition'' is satified when $g$ is dominated by a drift function as we saw in lemma~\ref{lemma:lindeberg}. 

The law of the iterated logarithm is given by corollary~4.2 and theorem~4.8 in~\cite{HH80} since the assumption is satisfied according to lemma~\ref{lemma:lindeberg}.
\end{proof}

\section{About the nullity of the variance} \label{subsection:nullite_variance}

\begin{miniabstract}
In this section, we study conditions under which the variance appearing in the central limit theorem and in the law of the iterated logarithm can not vanish.
\end{miniabstract}

Let $\G$ be a locally compact group acting continuously on a topological space $\X$ preserving the probability measure $\nu$.

We will always assume that the action of $\G$ on $\X$ is $\nu-$ergodic : this means that every measurable $\G-$invariant function is constant $\nu-$a.e.

\medskip
Let $\rho$ be a probability measure on $\G$ and $P$ the associated Markov operator on $\X$.

For any $f\in \mathrm{L}^2(\X,\nu)$, we have, using Jensen's inequality, that
\begin{flalign*}
\|Pf\|^2_2 &= \int_\X \left| \int_\G f(gx) \di\rho(g)\right|^2 \di \nu(x) \leqslant \int_\G \int_\X |f(gx)|^2 \di\nu(x) \di\rho(g) \\&\leqslant \int_\X  |f(x)|^2 \di\nu(x) = \|f\|_2^2 
\end{flalign*}
And so, the operator $P$ is continuous on $\mathrm{L}^2(\X,\nu)$ and $\|P\| \leqslant 1$. It is clear that $\|P\|=1$ since $P1=1$.

\medskip
In our study of the central limit theorem for some function $f$ on $\X$, the variance will always be given
\[
\sigma^2(f) = \|g\|^2_2 - \|Pg\|_2^2
\]
where $g \in \mathrm{L}^2(\X,\nu)$ is a function that we will have constructed such that $f-\int f\di\nu = g-Pg$ (in $\mathrm{L}^2(\X,\nu)$). It is therefore important to know if there can be some non-constant function $g\in \mathrm{L}^2(\X,\nu)$ such that $\|Pg\|_2 = \|g\|_2$.

\medskip
This question has been studied by Furman and Shalom in~\cite{FuSh99} where they prove that if the measure $\rho$ is aperiodic (that is to say that it's support is not included in a class of a subgroup of $\G$) then there is no non-constant function $f\in \mathrm{L}^2(\X,\nu)$ such that $\|Pf\|_2 = \|f\|_2$.

We prove in this section that the existence of such functions is equivalent to the existence of a subgroup $\Hb$ of $\G$ that does not act $\nu-$ergodically on $\X$ and of some $g\in \G$ such that $\supp\rho \subset \Hb g$. This will be our proposition~\ref{proposition:nullite_variance}.

\medskip
If $\rho$ is a borelian probability measure on $\G$, we note $\widetilde \rho$ the symmetrized measure. It is the probability measure defined for any borelian subset $A$ of $\G$ by
\[
\widetilde \rho(A) = \int_\G \un_A(g^{-1}) \di\rho(g)
\]

\begin{remark}
Since the measure $\nu$ is $\G-$invariant, we can compute, for any $f_1, f_2 \in \mathrm{L}^2(\X,\nu)$,
\begin{flalign*}
\int_\X f_2 P_{\rho}f_1 \di\nu &= \int_\G \int_\X f_1(gx) f_2(x) \di\nu(x) \di\rho(g) = \int_\G \int_\X f_1(x) f_2(g^{-1} x) \di\nu(x) \\ &= \int_\X f_1 P_{\tilde\rho}f_2 \di\nu
\end{flalign*}
So, the operator $P_{\tilde \rho}$ is the adjoint operator of $P_\rho$ in $\mathrm{L}^2(\X,\nu)$.
\end{remark}

\begin{remark} \label{remark:P_etoile_P}
In our definition of $P_{\rho}$, we make the element $g$ act on the left. Thus, if $\rho_1, \rho_2$ are borelian probability measures on $\G$, for any $f\in \mathrm{L}^2(\X,\nu)$ and any $x\in \X$, we get
\[
P_{\rho_1 } P_{\rho_2} f(x) = \int_\G P_{\rho_2 }f(g x) \di\rho_1(g) = \int_\G \int_\G f(g_2 g_1 x) \di\rho_1(g_1) \di\rho_2(g_2) = P_{\rho_2 \ast \rho_1} f(x)
\]
Thus, $P_{\rho_1} P_{\rho_2}$ is the operator associated to the measure $\rho_2 \ast \rho_1$. This inversion doesn't have any consequence in this article (since we always convol a measure with it's powers) but in this section we have to remember that the measure associated to $P^\ast P$ is $\rho\ast \tilde \rho$.
\end{remark}

\bigskip
First, we remark that for any $f\in \mathrm{L}^2(\X,\nu)$,
\[
\|f\|_2^2- \|Pf\|_2^2 = \int_\X f^2(y) - (Pf)^2(y) \di\nu(y) = \int_\X f(y) (I_d-P^\ast P)f(y) \di\nu(y) 
\]
where $P^\ast$ is the adjoint operator of $P$ in $\mathrm{L}^2(\X,\nu)$.

Moreover, we saw that $\|f\|^2 - \|Pf\|^2 \geqslant 0$.

\begin{lemma}
Let $\G$ be a group, $S\subset \G$ and $S^{-1} = \{g^{-1}|g\in S\}$.

Then, the subgroup of $\G$ generated by $SS^{-1} $ is the smallest subgroup $\Hb$ of $\G$ such that there is $g\in \G$ with $S \subset \Hb g$.
\end{lemma}

\begin{proof}
First, let $\Hb$ be a subgroup of $ \G$ and $g\in \G$. If $S\subset  \Hb g$ then $SS^{-1} \subset \Hb  g g^{-1} \Hb = \Hb$.

On the other hand, let $\Hb$ be a subgroup of $\G$ containing  $S S^{-1} $ and let $g\in S$.

Then, for any $h\in S$, we have that $h= h g^{-1} g$. But, $ h g^{-1} \in \Hb$ and so $h\in \Hb g$. This proves that $S\subset \Hb g$.

What we proved is that for any subgroup $\Hb$ of $\G$, we have the equivalence between ``$ S S^{-1} \subset \Hb$'' and ``there is $g\in \G$ such that $S\subset \Hb g$''. This proves the lemma since the subgroup of $\G$ generated by $ S S^{-1}$ is by definition the smallest subgroup of $\G$ containing $ SS^{-1} $.
\end{proof}

\begin{proposition} \label{proposition:nullite_variance}
Let $\G$ be a locally compact group acting continuously and ergodically on a topological space $\X$ endowed with a $\G-$invariant probability measure $\nu$.

Then, for any $f\in \mathrm{L}^2(\X,\nu)$, the three following assertions are equivalent
\begin{enumerate}
\item $\|Pf\|_2 = \|f\|_2$
\item For $\nu-$a.e. $x\in \X$ and $\rho\ast \tilde \rho-$a.e. $g\in \G$, $f(gx)=f(x)$.
\item There is some subgroup $\Hb$ of $\G$ and some $g\in \G$ such that $f$ is $\Hb-$invariant and $\supp \rho \subset \Hb g$.
\end{enumerate}
\end{proposition}

\begin{remark}
There can exist a non constant function $f\in \mathrm{L}^2(\X,\nu)$ such that $\|Pf\|_2 = \|f\|_2$ only if $\supp\rho$ is included in a right-class of a subgroup of $\G$ whose action on $\X$ is not $\nu-$ergodic.
\end{remark}

\begin{proof}
First, we remark that
\begin{flalign*}
\int_\X \int_\G |f(g x) - f(x)|^2 \di(\rho \ast \tilde\rho)(g) \di\nu(x) &= \int_\X 2|f(x)|^2 - 2 \Re( \overline{f(x)}P^\ast Pf(x)) \di\nu(x) \\
& = 2 \|f\|_2^2 - 2 \Re\left( \int \overline f P^\ast P f \di\nu \right) \\
&= 2 \|f\|^2_2 - 2 \|Pf\|_2^2
\end{flalign*}
So the first point implies the second one.

The second point implies that the function $f$ is invariant by the subgroup generated by $(\supp \rho) (\supp \rho)^{-1}$. But, according to the previous lemma, this subgroup is precisely the smallest subgroup $\Hb$ of $\G$ such that there is $g\in \G$ with $\supp \rho \subset \Hb g$. And so, the second point implies the third.

Finally, if there is some $g$ in $\G$ and a subgroup $\Hb$ such that $f$ is $\Hb-$invariant and $\supp\rho \subset \Hb g$, then, for $\nu-$a.e. $x\in \X$ and any $\gamma\in \supp\rho$, $f( \gamma x) = f(gx)$ and so,
\[
Pf(x) = \int_\G f(\gamma x) \di\rho(\gamma) = f(gx)
\]
Thus,
\[
\int_\X |Pf(x)|^2 \di\nu(x) = \int_\X |f(gx)|^2 \di\nu(x) = \int_\X |f(x)|^2\di\nu(x)
\]
And the third point implies the first one.
\end{proof}

\begin{corollary} \label{corollary:nullite_variance}
Let $\G$ be a locally compact group acting continuously and $\nu-$ergodically on a topological space $\X$ endowed with a $\G-$invariant probability measure $\nu$.

Let $\rho$ be a borelian probability measure on $\G$.

Let $g \in \mathrm L^2(\X,\nu)$ such that $\nu\left(\left\{ x\in \X\middle| \sup_n P^n g^2(x)<+\infty \right\}\right)=1$ and note $f=g-Pg$. Suppose that $\|g\|_2=\|Pg\|_2$ then, for $\nu-$a.e. $x\in \X$, the sequence $(\sum_{k=0}^n f(g_k \dots g_1 x))$ is bounded in $\mathrm{L}^2(\prob_x)$.

Moreover if $g$ belongs to $\mathrm{L}^\infty(\X)$ then, for $\nu-$a.e. $x\in \X$, we have that the sequence $(\sum_{k=0}^n f(g_k \dots g_1 x))$ is bounded in $\mathrm{L}^\infty(\prob_x)$.
\end{corollary}

\begin{proof}
According to the previous proposition, if $\|g\|=\|Pg\|$, there is some $\gamma \in \G$ and a subgroup $\Hb$ of $\G$ such that $\supp \rho \subset \Hb \gamma$ and $g$ is $\Hb-$invariant.

So, for $\nu-$a.e. $x\in \X$ and $\rho-$a.e. $g_1 \in \G$, $g(g_1 x) =g(\gamma x)$. In particular, $Pg(x) = g(\gamma x)$ and so, $f(x) = g(x) - g(\gamma x)$.

Thus, for $\nu-$a.e. $x\in \X$ and $\rho^{\otimes n}-$a.e. $(g_i) \in \G^n$,
\begin{flalign*}
\sum_{k=0}^{n-1} f(g_k \dots g_1 x) &= \sum_{k=0}^{n-1} g(g_k \dots g_1 x) - g(\gamma g_k \dots g_1 x) \\&= g(x) - g(g_n \dots g_1 x) + \sum_{k=0}^{n-1} g(g_{k+1} \dots g_1 x) - g(\gamma g_k \dots g_1 x) \\&= g(x) - g(g_n \dots g_1 x)
\end{flalign*}
This computation proves the corollary when the function $g$ is bounded.

Moreover, we have that
\begin{flalign*}
 \int_{\G^\N} \left| \sum_{k=0}^{n-1} f(g_k \dots g_1 x)\right|^2 \di\rho^{\otimes \N}((g_i)) &= g(x)^2 + P^n(g^2)(x) - 2 g(x)P^n g(x) \\&\leqslant g(x)^2 + P^n (g^2)(x) + 2 |g(x)| \sqrt{ P^n (g^2)(x)} \\
 & \leqslant 4\sup_n P^n (g^2)(x)
\end{flalign*}
Where we used Jensen's inequality to say that $|P^n g(x)| \leqslant \sqrt{P^n g^2(x)}$.

This finishes the proof of the corollary.
\end{proof}

The following example is an illustration of the previous corollary in an explicit context.
\begin{example}
Let
\[
A=\left(\begin{array}{cc} 2 & 1 \\ 1 & 1\end{array}\right)\text{ et }B =\left(\begin{array}{cc} 0 & 1 \\ -1 & 0\end{array}\right)
\]

Then, the subgroup of $\mathrm{SL}_2(\R)$ generated by $A$ and $B$ is Zariski-dense and the Lebesgue measure $\nu$ on the torus $\T^2=\R^2/\Z^2$ is ergodic.

Let $\rho = \frac 1 2 \delta_A + \frac 1 2 \delta_{BA}$.

Guivarc'h proved in \cite{Gu06} that the operator $P$ associated to $\rho$ has a spectral gap in $\mathrm{L}^2(\T^2, \nu)$.

Let $\|\,.\,\|$ be the distance induced on $\T^2 = \R^2/\Z^2$ by the euclidean norm on $\R^2$.
And let $g$ be the function defined for any $x\in \T^2$ by $g(x) = \|x\|$.

Then, for any $x\in \T^2$,
\[
Pg(x) = \frac 1 2 \|Ax\| + \frac 1 2 \|BA x\| = \|Ax\|= g(Ax)
\]
and
\[
\int_\X |Pg(x)|^2 \di\nu(x) = \int_\X |g(Ax)|^2 \di\nu(x) = \int_\X |g(x)|^2 \di\nu(x)
\]
Moreover, if we note $f=g-Pg$, then, for any $x\in \X$, $n\in \N$ and any $(g_1, \dots g_n) \in \{A, BA\}^n$, we have that
\[
g(g_{n+1} \dots g_1 x) =g(A g_n \dots g_1x)
\]
and so,
\begin{flalign*}
\sum_{k=0}^{n-1} f(g_k \dots g_1 x) &= g(x) - g(g_n \dots g_1 x) + \sum_{k=0}^{n-1} g(g_{k+1} \dots g_1 x) - g(A g_k \dots g_1x) \\&= g(x) - g(g_n \dots g_1 x)
\end{flalign*}
This proves that for any $x\in \X$, the sequence $(\sum_{k=0}^{n-1} f(g_k \dots g_1 x))$ is bounded in $\mathrm{L}^\infty (\prob_x)$.
\end{example}
\section{Application to the random walk on the torus} \label{section:tore_point_diophantien}

\begin{miniabstract}
In this section, we go back to the random walk on the torus. The law of large numbers is known as a corollary of a theorem in~\cite{BFLM11} which allow one to have a speed of convergence depending on the diophantine properties of the starting point. We use this to prove the central limit theorem and the law of the iterated logarithm.
\end{miniabstract}

Let $\Hb$ be a subgroup of $\mathrm{SL}_d(\R)$.
We say that the action of $\Hb$ on $\R^d$ is \emph{strongly irreducible} if $\Hb$ doesn't fixe any finite union of proper subspaces of $\R^d$ and that it is \emph{proximal} if for some $h\in \Hb$ we have a decomposition $\R^d = V_h^+ \oplus V_h^<$ of $\R^d$ into an $h-$invariant line $V_h^+$ and an $h-$invariant hyperplane $V_h^<$ such that the spectral radius of $h$ restricted to $V_h^<$ is strictly smaller than the one of $h$ restricted to $V_h^+$.

We say that the group $\Hb$ is strongly irreducible and proximal if it's action is.

\medskip
If we also assume that $\Hb $ is a subgroup of $ \mathrm{SL}_d(\Z) $, then it's action pass to the quotient $\X:=\T^d = \R^d/\Z^d$ that we endow with a metric defined by a norm on $\R^d$ and with Lebesgue's measure $\nu$. Moreover, $\Hb$ is strongly irreducible and proximal then any $a\in \Z^d \setminus\{0\}$ has an infinite $\Hb-$orbit and so, according to the proposition~1.5 in~\cite{BM00} the action of $\Hb$ on $\T^d$ is $\nu-$ergodic (every $\Hb-$invariant function is constant $\nu-$a.e.).

\medskip
Let $\rho$ be a probability measure on $\mathrm{SL}_d(\Z)$. We define a random walk on $\X$ noting, for $x\in \X$,
\[
\left\{\begin{array}{ccl} X_0 & = & x \\ X_{n+1} &=& g_{n+1} X_n \end{array} \right.
\]
where $(g_n) \in \mathrm{SL}_d(\Z)^\N$ is an iid sequence of random variables of common law $\rho$.

\medskip
In this constext, Bourgain, Furmann, Lindenstrauss and Mozes proves the following
\begin{theorem*}[\cite{BFLM11}] \label{theorem:BFLM}
Let $\rho$ be a borelian probability measure on $\mathrm{SL}_d(\Z)$ whose support generates a strongly irreducible and proximal group and which has an exponential moment\footnote{There is $\varepsilon\in \R_+^\ast$ such that
\[
\int_{\mathrm{SL}_d(\Z)} \|g\|^\varepsilon \di \rho(g) \text{ is finite}\]}.

Note
\[
\lambda_1 = \int_{\mathrm{SL}_d(\Z)} \int_{\prob(\R^d)} \ln \|gx\| \di\nu(x) \di \rho(g) >0
\]
where $\nu$ is the unique\footnote{The fact that $\lambda_1$ exists and is strictly non negative and that $\nu$ exists and is unique comes from a result in~\cite{GuRa85}.} $\rho-$stationary probability measure on $\prob(\R^d)$.

Then, for any $\varepsilon \in \R_+^\ast$, there is a constant $C$ such that for any $x\in \T^d$, any $a\in \Z^d\setminus\{0\}$, any $t\in ]0,1/2]$ and any $n\in \N$ with $n\geqslant -C\ln t$, if
\[
|\widehat{\rho^{\ast n}\ast \delta_x }(a)| > 2t \|a\| 
\]
then, $x$ admits a rational approximation $p/q \in \Q^d/\Z^d$ satisfying
\[
d\left(x,\frac p q\right)\leqslant e^{-(\lambda_1- \varepsilon) n} \text{ and } |q| \leqslant t ^{-C}
\]
\end{theorem*}

In particular, this proves that if $x$ is irrational, then, for any $a\in \Z^d \setminus \{0\}$ and any $t\in ]0,1/2]$, there are only finitely many $n\in \N$ such that $|\widehat{\rho^{\ast n}\ast \delta_x }(a)| > 2t \|a\| $. This proves that for any irrational point $x$ in $\T^d$ and any $a\in \Z^d \setminus\{0\}$, 
\[
\lim_{n\to +\infty} \widehat {\rho^{\ast n} \ast \delta_x}(a)  =0
\]
And so, using Weyl's equidistribution criterion we have that for any continuous function $f$ on $\T^d$ and any irrational point $x\in \T^d$,
\begin{equation}\label{equation:BFLM}
\lim_{n\to +\infty} P^nf(x) = \int f\di\nu
\end{equation}
where $\nu$ is the Lebesgue measure on $\T^d$. Moreover, the speed of convergence depend on the diophantine properties of $x$ (see corollary $C$ in~\cite{BFLM11}).
In this section, we first want to obtain a more explicit speed of convergence in equation~\ref{equation:BFLM} in terms of diophantine properties of $x$. Then, we want to use this speed of convergence to prove the central limit theorem and the law of the iterated logarithm for starting points having good diophantine properties.

\medskip
In the first sub-section, we state a corollary of the theorem~\ref{theorem:BFLM} that is easier to deal with. The price we have to pay is that we will only be able to study hölder continuous functions. This will be proposition
\begin{proposition*}[\ref{proposition:vitesse_convergence_tore}]
Let $\rho$ be a borelian probability measure on $\mathrm{SL}_d(\Z)$ whose support generates a strongly irreducible and proximal group and which has an exponential moment.

Then for any $\gamma, \delta\in ]0,1]$ and any strictly non-decreasing function $\varphi : \R_+ \to \R_+^\ast$ with
\[
\liminf \frac{ \ln \varphi(s)}{\ln s} >0
\]
there are constants $C,C_0,C_1\in \R_+^\ast$such that for any $x\in \T^d$ and any $n\in \N$,
\[
\cal W_\gamma (\rho^{\ast n}\ast \delta_x , \nu) \leqslant C\psi(n) h_\varphi(x)^\delta
\]
where $h_\varphi$ is the function defined for any $x\in \T^d$ by
\[
h_\varphi(x) = \sup_{p/q \in \Q^d/\Z^d} \frac 1 {\varphi(q) d(x,p/q)}
\]
the function $\psi$ is defined by
\[
\psi(t) = \left(\varphi^{-1} (e^{C_1t}) \right)^{-C_0}
\]
and $\cal W_\gamma$ is the Wasserstein distance defined for any probability measure $\vartheta_1, \vartheta_2$ on the torus by
\[
\cal W_\gamma(\vartheta_1, \vartheta_2) = \sup_{\substack{ f\in \cal C^{0,\gamma}(\T^d) \\ \| f\|_\gamma \leqslant 1}} \left| \int_\X f\di\vartheta_1 - \int_\X f\di\vartheta_2 \right|
\]
Where $\cal C^{0,\gamma}(\T^d)$ and $\|f\|_\gamma$ where defined in equation~\ref{equation:norme_gamma}.
\end{proposition*}
Then, we will prove that there is a function $u_\varphi$ that dominates the function $h_\varphi$ and such that $Pu_\varphi \leqslant au_\varphi + b$ for some $a\in ]0,1[$ and $b\in \R$. This means that in average, $u_\varphi(gx)$ is much smaller than $u_\varphi(x)$ and this will allow us to prove, using the results of section~\ref{section:LLN_TCL}, the
\begin{theorem*}[\ref{theorem:TCL_tore_diophantien}]
Let $\rho$ be a borelian probability measure on $\mathrm{SL}_d(\Z)$ whose support generated a strongly irreducible and proximal group and which has an exponential moment.

Then, for any  $\gamma \in ]0,1]$ there is $\beta_0 \in \R_+^\ast$ such that for any $B \in \R_+^\ast$ and any $\beta \in ]0,\beta_0[$ we have that for any irrational point $x\in \T^d$ such that the inequality
\[
d\left(x,\frac pq \right) \leqslant e^{-Bq^\beta}
\]
has only finitely many solutions $\frac pq \in \Q^d / \Z^d$, we have that for any $\gamma-$holder continuous function $f$ on the torus, noting $\sigma^2(f)$ the quantity defined in equation~\ref{equation:variance} we have that
\[
\frac 1 {\sqrt n} \sum_{k=0}^{n-1} f(X_k) \xrightarrow{\mathcal{L}} \cal N\left(\int f\di \nu, \sigma^2(f) \right)
\]
(If $\sigma^2=0$, the law $\cal N(\mu,\sigma^2)$ is a Dirac mass at $\mu$).

Moreover, if $\sigma^2(f) \not=0$ then
\[
\liminf \frac{ \sum_{k=0}^{n-1} f(X_k) - \int f \di \nu}{\sqrt{2n\sigma^2(f) \ln\ln n}} =-1 \text{ et }\limsup \frac{ \sum_{k=0}^{n-1} f(X_k) - \int f \di \nu}{\sqrt{2n\sigma^2(f) \ln\ln n}} =1 
\]
and if $\sigma^2(f)=0$, then for $\nu-$a.e. $x\in \X$, the sequence $ (\sum_{k=0}^{n-1} f(X_k) - \int f \di \nu)_n$ is bounded in $\mathrm{L}^2(\prob_x)$.
\end{theorem*}

\subsection{BFLM's result for holder-continuous functions} \sectionmark{BFLM for holder-continuous functions}

\begin{miniabstract}
In this section, we start with a few remind on Wasserstein's distance and then we state BFLM's result using this distance.\end{miniabstract}

\subsubsection{Wasserstein's distance on the torus}
Note $\X$ the torus $\T^d = \R^d/\Z^d$ endowed with the metric induced by a norm on $\R^d$.

If $\vartheta_1$ and $\vartheta_2$ are borelian probability measures on $\X$, a way to measure their distance is to compute the total variation
\[
d_{\mathrm{var}}(\vartheta_1, \vartheta_2) = \sup_{\substack{f\in \cal C^0(\X) \\ \|f\|_\infty\leqslant 1}} \left|\int f\di\vartheta_1 - \int f\di \vartheta_2 \right|
\]
This distance is not adapted to our study since, for instance, when $\rho$ has a finite support, so does the measure $\rho^{\ast n} \ast \delta_x$ and so, for any $x\in \X$ and any $n\in \N$,
\[
d_{\mathrm{var}}(\rho^{\ast n} \ast \delta_x, \nu) = 2
\]

However, we can compute the distance between $\vartheta_1$ and $\vartheta_2$ seen has linear forms on the space $\cal C^{0,\gamma}(\X)$ of $\gamma-$holder continuous functions on $\X$. Therefore, we make the following

\begin{definition}[Wasserstein's distance]~

Let $\vartheta_1,\vartheta_2$ be two borelian probability measures on a compact metric space $(\X,d)$.

For any $\gamma\in]0,1]$, we define the $\gamma-$distance of Wasserstein between $\vartheta_1$ and $\vartheta_2$ by
\[
\cal W_\gamma (\vartheta_1,\vartheta_2) = \sup_{f\in\cal C^{0,\gamma}(\X)\;\|f\|_\gamma\leqslant 1} \left|\int f\di\vartheta_1-\int f \di\vartheta_2 \right|
\] 
\end{definition}

\begin{remark}
Sometimes, this distance is also named after Kantorovich and Rubinstein and we refer to~\cite{Vil09} for an overview of it's first properties.
\end{remark}

On the torus, Wasserstein's distance between a given measure $\vartheta$ and Lebesgue's measure is linked to the decreasing of the Fourier coefficients of $\vartheta$. We make this precise in next
\begin{lemma}\label{lemma:wasserstein_distance_torus}
For any $\gamma \in ]0,1]$, there is a constant $C$ depending only on $d$ and $\gamma$ such that for any borelian probability measure $\vartheta$ on the torus $\T^d$ and any $t\in\R_+^\ast$, if $\cal W_\gamma(\vartheta,\nu)  > t$ then there is $a\in \Z^d\setminus\{0\}$ such that $|\widehat{\vartheta}(a)|\geqslant Ct^{C}\|a\|$ where we noted $\nu$ the Lebesgue measure on $\T^d$.
\end{lemma}

To prove this lemma, we will need a result of Jackson and Bernstein about the rate at which one can approximate in the uniform norm an holder continuous function by more regular ones.

For $r\in \N^\ast$, we define the Sobolev space
\[
\cal H^r := \left\{ f\in \mathrm{L}^2(\T^d) \middle| \sum_{a\in \Z^d} |\widehat f(a)|^2 (1+\|a\|)^{2r} <+\infty \right\}
\]

\begin{lemma}[Jackson, Bernstein]\label{lemma:interpolation}
Let $\gamma\in]0,1]$ et $r\in[1,+\infty[$.

Then, there is some $C\in\R$ such that for any function $f\in\cal C^{0,\gamma}(\T^d)$, there is a sequence $(f_n)\in\cal H^r(\T^d)^\N$ such that for any $n\in\N^\ast$,
\[
\int f\di \nu=\int f_n\di \nu,\;\;
\|f-f_n\|_\infty \leqslant \frac{C}{n^\gamma} \|f\|_\gamma\textnormal{ and }\|f_n\|_{\cal H^r} \leqslant C\|f\|_\infty n^{C}
\]
\end{lemma}

\begin{proof}
For $y\in\R/\Z$, we note $k_m(y)=\left(\frac{\sin(2\pi my)}{ \sin(\pi y)}\right)^4$ and for a point $y=(y_1,\dots,y_d)\in\T^d$, we note $K_m(y)=\prod_{i=1}^d k_m(y_i)$. Finally, we note $I_m=\left(\int_{-1/4}^{1/4} k_m(y)\di y\right)^{-1}$.

Define, for $x\in \T^d$,
\[
f_m(x)=\int_{[-1/4,1/4]^d} I_m^d f(x+2y)K_m(y)\di y=\frac{I_m^d}{2} \int_{[-1/2,1/2]^d} f(y) K_m(\frac{y-x} 2)\di y
\]

Then, we can compute
\begin{flalign*}
|f(x)-&f_m(x)| = \left|\int_{[-1/4,1/4]^d} I_m^d (f(x)-f(x+2y))K_m(y)\di y \right|\\
&\leqslant I_m^d 2^\gamma\|f\|_{\gamma} \int_{[-1/4,1/4]^d}  \|y\|^\gamma K_m(y)\di y\leqslant I_m^d 2^{1+\gamma}\|f\|_{\gamma} \int_{[0,1/4]^d}  \|y\|^\gamma K_m(y)\di y \\
&\leqslant I_m^d 2^{1+\gamma}\|f\|_{\gamma} \int_{[0,1/4]^d} ( y_1^\gamma+\dots+y_d^\gamma) K_m(y)\di y \\ &\leqslant dI_m 2^{1+\gamma}\|f\|_\gamma \int_{[0,1/4]} y^\gamma k_m(y)\di y
\end{flalign*}
Where we used in last inequality the fact that
\[
I_m^d \int_{[0,1/4]^d} y_1^\gamma K_m(y) \di y = I_m \int_{[0,1/4]} y^\gamma k_m(y) \di y
\]
Note now,
\[
J_{m,\gamma} := 2\int_{0}^{1/4} y^\gamma k_m(y) \di y = 2\int_{0}^{1/4} y^{\gamma} \left(\frac{\sin (2\pi m y)}{\sin (\pi y)}
\right)^4 \di y\]
Then, using that for any $t \in [0,\pi /2]$, $\frac{2t}{\pi } \leqslant \sin(t) \leqslant t$, we get that
\[
\frac 1 {\pi ^4} \int_{0}^{\pi/4} y^{\gamma - 4}\left( \sin(2\pi m y)  \right)^4 \di y\leqslant J_{m,\gamma} \leqslant \frac 1 {2^4} \int_{0}^{\pi/4} y^{\gamma-4} \left( \sin(2\pi m y)  \right)^4 \di y := \frac 1 {2^4} L_{m,\gamma}
\] 
Moreover,
\[
L_{m,\gamma} = \int_{0}^{m\pi/2} \left(\frac y{2\pi m}\right)^{\gamma-4} \left(\sin y\right)^4 \frac{\di y}{2\pi m} = (2\pi m)^{3-\gamma} \int_0^{m\pi /2} y^{\gamma-4}\left( \sin y\right)^4 \di y
\]
And so,
\[
L_{m,\gamma} \asymp m^{3-\gamma} 
\]
Thus,
\[
J_{m,\gamma} \asymp m^{3-\gamma}
\]
and finally,
\[
I_m \int_{0}^{1/4} y^{\gamma} k_m(y) \di y =\frac{J_{m,\gamma}}{J_{m,0}} \asymp m^{-\gamma}
\]
And so, what we proved is that there is some constant $C$ such that for any function $f\in \cal C^{0,\gamma}(\T^d)$, we have that
\[
\|f-f_m\|_{\infty} \leqslant \frac{C}{m^\gamma} \|f\|_\gamma
\]

So, what is left is to prove that (for some maybe bigger constant $C$)
\[
\|f_m\|_{\cal H^r} \leqslant C \|f\|_\infty m^{rd}
\]
But, it is clear that for any $a \in \Z^d$,
\[
|\widehat f_m(a)|\leqslant \|f\|_\infty
\]
And, using that $f_m = f\ast K_m$ and that $K_m$ is a trigonometric polynomial of degree at most $Cm^4$ for some  $C$ as we may see by developping
\begin{flalign*}
k_m (y) &= \left(\frac{\sin(2\pi m y)}{\sin (\pi y)} \right)^4 = \left( \frac{e^{-2i\pi m y} - e^{2i\pi my}}{e^{-i\pi y} - e^{i\pi y}}\right)^4 = e^{4i\pi y} \left(\frac{e^{-2i\pi my }- e^{2i\pi m y}}{1 - e^{2i\pi y}} \right)^4 \\
&= e^{4 i\pi y}\left( \sum_{k=-m}^{m-1} e^{2i\pi k y}\right)^4
\end{flalign*}
So, we have that for $\|a\| > Cm^4$, $\widehat K_m(a) = 0$.

And this proves that
\begin{flalign*}
\|f_m\|_{\cal H^r} &= \left( \sum_{a\in \Z^d} (1+\|a\|)^{2r} |\widehat f_m(a)|^2 \right)^{1/2}\leqslant \left(\sum_{\|a\|\leqslant Cm^4} (1+\|a\|)^{2r}\right)^{1/2} \|f\|_\infty \\
& \leqslant (1+ Cm^4)^r (Cm^4 )^{d/2}\|f\|_\infty
\end{flalign*}
Which finishes the proof of the lemma.
\end{proof}

\begin{proof}[Proof of lemma~\ref{lemma:wasserstein_distance_torus}]
By definition of $\cal W_\gamma (\vartheta,\nu)$, there is a function $f\in\cal C^{0,\gamma}(\T^d)$ such that $\|f\|_\gamma\leqslant 1$ and $|\int f\di\vartheta -\int f\di \nu|\geqslant \frac t2$.

\medskip
Let $r\in \N^\ast$ such that $\sum_{a\in\Z^d\setminus\{0\}} \frac{\|a\|}{(1+\|a\|^2)^{r/2}}=:C_r$ is finite.

According to lemma~\ref{lemma:interpolation}, there is a sequence of functions $(f_n)\in\cal H^r(\T^d)^\N$ such that $\|f-f_n\|_\infty \leqslant \frac{C}{n^\gamma}$ and $\|f_n\|_{\cal H^r} \leqslant C n^{C}$.

Then,
\begin{flalign*}
\left|\int f_n\di\vartheta -\int f_n \di m\right| & \geqslant \left|\int f\di\vartheta -\int f \di m\right| - \left|\int (f-f_n)\di\vartheta -\int (f-f_n) \di m \right| \\
& \geqslant \frac t 2 - 2\|f-f_n\|_\infty \geqslant \frac t2-\frac {2C}{n^\gamma}
\end{flalign*}
But,
\begin{flalign*}
\left|\int f_n\di\vartheta -\int f_n \di m\right| & = \left| \sum_{a\in\Z^d \setminus\{0\}} \widehat{f_n}(a) \widehat{\vartheta}(a) \right| \leqslant \sum_{a\in\Z^d \setminus\{0\}} \left| \widehat{f_n}(a) \widehat{\vartheta}(a) \right| \\
&\leqslant \sum_{a\in\Z^d\setminus\{0\}} \frac{\|f_n\|_{\cal H^r}}{(1+\|a\|^2)^{r/2}} |\widehat{\vartheta}(a)| \leqslant \|f_n\|_{\cal H^r} C_r \sup_{a\in \Z^d \setminus\{0\}} \frac{|\widehat \vartheta (a)|}{\|a\|} \\
&\leqslant Cn^C C_r \sup_{a\in \Z^d \setminus\{0\}} \frac{|\widehat \vartheta (a)|}{\|a\|}
\end{flalign*}
and so,
\[
\sup_{a\in \Z^d \setminus\{0\}} \frac{|\widehat \vartheta(a)|}{\|a\|} \geqslant \frac{\frac t2-\frac{2C}{n^\gamma} }{C C_r n^C}
\]

So, taking $n=\lfloor \left(\frac{8C}t\right)^{1/\gamma}  \rfloor +1$ we have that $\frac t 2-2\frac{C_1}{n^\gamma} \geqslant t/4$ and there is some constant $C'$ such that
\[
\sup_{a\in \Z^d\setminus\{0\}} \frac{|\widehat \vartheta(a)|}{\|a\|} \geqslant C' t^{1 + C/\gamma} 
\]
and this finishes the proof.
\end{proof}

With lemma~\ref{lemma:wasserstein_distance_torus}, we get a straightforward corollary of theorem~\ref{theorem:BFLM}.
\begin{proposition}[\cite{BFLM11} with Wasserstein's distance] \label{proposition:BFLM_2}
Let $\rho$ be a borelian probability measure on $\mathrm{SL}_d(\Z)$ whose support generated a strongly irreducible and proximal group and which has an exponential moment.

Then, for any $\varepsilon \in \R_+^\ast$ and any $\gamma \in ]0,1]$, there is a constant $C\in \R_+$ and $t_0 \in ]0,1/2]$ such that for any $n\in \N$, any $t\in ]0,t_0]$ with $n\geqslant - C\ln t$ and any $x\in \T^d$, if
\[
\cal W_\gamma(\rho^{\ast n}\ast \delta_x,\nu) \geqslant t
\]
then there is $p/q\in \Q^d/\Z^d$ with $|q| \leqslant Ct^{-C}$ and
\[
d(x,p/q) \leqslant e^{-(\lambda_1-\varepsilon)n}
\]
\end{proposition}

The previous proposition proves that if the distance between $\rho^{\ast n} \ast \delta_x$ and $\nu$ is large and if $t$ is a function of $n$, then $x$ is well approximated by rational points : for instance, if $t=e^{-\alpha n}$ for some $\alpha \in \R_+^\ast$ then the $p/q$ produced by the proposition satisfies $q\leqslant C e^{\alpha C n}$ and so,
\[
d(x,p/q) \leqslant e^{-(\lambda_1-\varepsilon) n} \leqslant \left( \frac C q \right)^{(\lambda_1 - \varepsilon)/\alpha C}
\]
We are going to reverse this to, given a diophantine condition, find a rate of convergence.

\medskip
From now on, we fix a strictly non-decreasing function $\varphi : \R_+ \to \R_+^\ast$.

For $x\in \X$, we note
\begin{equation} \label{equation:hauteur_phi}
h_\varphi(x) = \sup_{p/q\in \Q^d/\Z^d} \frac{ 1 }{\varphi(q) d(x,p/q)}
\end{equation} 

Thus, a point is $M-$diophantine if $h_\varphi(x)$ is finite with $\varphi(t) = t^M$. We also remark that if $\varphi$ grows faster than any polynomial, then $\nu(h_\varphi<+\infty)=1$. 

\begin{proposition} \label{proposition:vitesse_convergence_tore}
Let $\rho$ be a borelian probability measure on $\mathrm{SL}_d(\Z)$ whose support generated a strongly irreducible and proximal group and which has an exponential moment.

Then, for any $\gamma, \delta\in ]0,1]$ and any strictly non-decreasing function $\varphi : \R_+ \to \R_+^\ast$ with
\[
\liminf \frac{ \ln \varphi(s)}{\ln s} >0
\]
there are constants $C,C_0,C_1\in \R_+^\ast$ such that for any $x\in \T^d$ and any $n\in \N$,
\[
\cal W_\gamma (\rho^{\ast n}\ast \delta_x , \nu) \leqslant C\psi(n) h_\varphi(x)^\delta
\]
where $h_\varphi$ is the function defined in equation~\ref{equation:hauteur_phi} and $\psi$ is the function defined by
\[
\psi(t) = \left(\varphi^{-1} (e^{C_1t}) \right)^{-C_0}
\]
\end{proposition}

\begin{remark}
The assumption on $\varphi$ implies that for some $c\in \R_+^\ast$ we have that for any $t\in \R$, $\varphi(t) \geqslant c t^c$. It is not restrictive at all since according to Dirichlet's theorem on diophantine approximation, if $\varphi(t) = o(t^{1+1/d})$, then the function $h_\varphi$ only takes infinite values.
\end{remark}

\begin{remark}
If we take $\varphi(n) = n^D$, then we get $\psi(n) = e^{-\kappa n}$ for some $\kappa \in \R_+^\ast$ and this proves that for a generic diophantine point, the convergence is at exponential speed.

\medskip
In the sequel, we will have to be sure that the sum of the $\psi(n)$ converges and so, we will take $\psi(n) = n^{-1-\alpha} $ for some $\alpha\in \R_+^\ast$. This will allow us to study irrational points $x\in \T^d$ such that the inequality
\[
d\left( x, \frac p q\right) \leqslant e^{-B q ^\beta}
\]
has only finitely many solutions $\frac pq\in \Q^d/\Z^d$ where $\beta,B$ will be constants depending on $\rho$.
\end{remark}

\begin{proof}
Let $C_0,C_1,C \in [5,+\infty[$ whose values will be determined later.

We note $C_2$ the constant given by proposition~\ref{proposition:BFLM_2}.

Let $x\in \X$ et $n\in \N$.

If $C\psi(n)h_\varphi(x)^\delta \geqslant 2$, then the inequality is satisfied since $\|P^{n}f-\int f\di m\|_\infty \leqslant 2\|f\|_\infty$.

Thus, we shall assume that $C\psi(n)h_\varphi(x)^\delta \leqslant 2$.

Let $t=\frac {C} 5 \psi(n)h_\varphi(x)^\delta$. Then, $t<\frac 1 2$ and
\[
-C_2\ln t = -C_2 \ln \left(\frac {C} 5 \psi(n)h_\varphi(x)^\delta \right) \leqslant -C_2 \ln \left(\psi(n) \right) = C_2 C_0\ln \varphi^{-1}(e^{C_1 n})
\]
since $C h_\varphi(x) / 5 \geqslant 1$ because $C\geqslant 5$ and $h_\varphi(x)\geqslant 1$.

But, there is a constant $C_3$ such that for any $s\in \R_+$, $\varphi(s) \geqslant C_3 s^{C_3}$ and so,
\[
\varphi^{-1}(s) \leqslant \left(\frac{ s} {C_3}\right)^{1/C_3}
\]
Therefore, $\ln \varphi^{-1}(e^{C_1n}) \leqslant \frac 1{C_3} (C_1 n-\ln C_3)$ and $-C\ln t \leqslant n$ if $C_0$ is small enough (depending on $C_1$). 

Thus, we can apply proposition~\ref{proposition:BFLM_2} to find that if
\[
\cal W_\gamma( \rho^{\ast  n} \ast \delta_x, \nu)\geqslant t
\]
then there is $p/q \in \Q^d/\Z^d$ with $q \leqslant C_2t^{-C_2}$ such that
\[
d\left(x,\frac p q\right) \leqslant e^{-\lambda n}
\]
Thus, as we shall assume without any loss of generality that $C_2 \left(\frac{5}{C} \right)^{C_2} \leqslant 1$ and $C_0 C_2 \leqslant 1$, we get that
\begin{flalign*}
q& \leqslant C_2 \left(\frac{5}{C\psi(n)h_\varphi(x)^\delta} \right)^{C_2} \leqslant C_2 \left(\frac{5}{C\psi(n)} \right)^{C_2} = C_2 \left(\frac{5}{C} \right)^{C_2} \left(\varphi^{-1}(e^{C_1n})\right)^{C_2C_0}  \\&\leqslant \varphi^{-1}(e^{C_1 n})
\end{flalign*}
and
\[
 e^{ \lambda n }\leqslant \|x-p/q\|^{-1} \leqslant  \varphi(q)h_\varphi(x) \leqslant \varphi(q) \left( \frac 2 {C \psi(n)} \right)^{1/\delta}  \leqslant e^{C_1 n}\left( \frac 2 {C \psi(n)} \right)^{1/\delta}
\]
Thus,
\[
\psi(n) \leqslant \frac 2 C e^{-\delta(\lambda-C_1)n}
\]
but,
\[
\psi(n) \geqslant \left( \frac{ e^{C_1 n}}{C_3} \right)^{-C_0}
\]
Which leads to a contradiction if $C_1< \lambda$, $C_0$ is small enough and $C$ is large enough.

Thus, there is no $n\in \N$ and $x\in \X$ such that $C \psi(n) h_\varphi(x)^\delta \leqslant 2$ and
\[
\cal W_\gamma (\rho^{\ast n}\ast \delta_x,\nu) \geqslant \frac 1 5 C \psi(n) h_\varphi(x)^\delta
\]
So, for any $n\in \N$ and any $x\in \T^d$,
\[
\cal W_\gamma (\rho^{\ast n}\ast \delta_x,\nu) \leqslant C \psi(n) h_\varphi(x)^\delta
\]
which is what we intended to prove.
\end{proof}

\subsection{Diophantine control along the walk}

\begin{miniabstract}
In this section, we are going to prove that if $x\in \X$ satisfies a diophantine condition, then so does the $gx$ with $g\in \mathrm{SL}_d(\Z)$. We will deduce from this a control of the speed of convergence in proposition~\ref{proposition:BFLM_2} along the walk.
\end{miniabstract}

We saw in proposition~\ref{proposition:vitesse_convergence_tore} that for any irrational point $x$ of $\T^d$, $\rho^{\ast n}\ast \delta_x$ converges for Wasserstein's distance to Lebesgue's measure on the torus. Moreover, the rate depend on the way $x$ can be approximated by rational points of the torus.

To prove the central limit theorem starting at some point $x$, we will have to control the rate of convergence of $\rho^{\ast n} \ast \delta_y$ for any $y$ of $\G x$ ; the problem being that the function $h_\varphi$ that we defined may take arbitrarily large values on $\G x$.

However, the set of points where $h_\varphi$ is finite is invariant under the action of $\Gamma=\mathrm{SL}_d(\Z)$ as one may see noting that for $x\in \T^d$, $p \in \Q^d/\Z^d$ and $g\in \Gamma$ we have
\[
\|g\| d(x,g^{-1} p) \geqslant d(gx,p) = d(gx,gg^{-1} p) \geqslant \frac 1 {\|g^{-1}\|} d(x,g^{-1} p)
\]
and $g^{-1} p $ is a rational point with the same denominator than $p$ (since $g^{-1}$ has integer coefficients) and this estimation proves that for any $g\in \mathrm{SL}_d(\Z)$ and any $x\in \T^d$,
\[
h_\varphi(gx) \leqslant \|g\| h_\varphi(x)
\]
In this section, we are going to prove that we can obtain a control that is far better than this trivial one. We will indeed prove that for any irrational point $x$ of the torus, in average, $gx$ is further from the rationals than $x$.

To do so, we begin by showing that, in average, $gx$ is further from $0$ than $x$. We will prove this in proposition~\ref{proposition:recurrence_loin_0} but at first, we will need the next
\begin{lemma}
Let $\rho$ be a borelian probability measure on $\mathrm{SL}_d(\Z)$ whose support generates a strongly irreducible and proximal group and which has an exponential moment.

For any $\delta \in \R_+^\ast$ and $x\in \T^d \setminus\{0\}$, we note
\[
u_\delta(x) = \frac 1 {d(x,0)^\delta}
\]
Then, there are $n_0\in \N$, $\delta\in \R_+^\ast$, $a\in [0,1[$ and $b\in \R$ such that for any $x\in \T^d \setminus\{0\}$,
\[
P^{n_0} u_\delta(x) \leqslant au_\delta(x) + b
\]
\end{lemma}

\begin{proof}
The proof is by going back to $\R^d$ since our assumptions imply that the first Lyapunov exponent is strictly non negative.

Let $\varepsilon \in \R_+^\ast$ and $\overline x\in B(0,\varepsilon) \subset \T^d$. Choose a point $x$ in $B(0,\varepsilon) \subset \R^d$ whose projection on the torus is $\overline x$. Then, for any $n\in \N$,
\begin{flalign*}
P^n u_\delta(\overline{x}) &= \int_\G d(g\overline{x},0)^{-\delta} \di\rho^{\ast n}(g) \\
&= \int_\G \un_{\|g\|\leqslant \frac 1 \varepsilon} d(g\overline{x},0)^{-\delta} \di\rho^{\ast n}(g)  + \int_\G \un_{\|g\|> \frac 1 \varepsilon} d(g\overline{x},0)^{-\delta} \di\rho^{\ast n}(g) \\
&= \int_\G \un_{\|g\|\leqslant \frac 1 \varepsilon} \|gx\|^{-\delta} \di\rho^{\ast n}(g) + \int_\G \un_{\|g\|> \frac 1 \varepsilon} d(g\overline{x},0)^{-\delta} \di\rho^{\ast n}(g) \\
& \leqslant \int_\G \|gx\|^{-\delta} \di\rho^{\ast n}(g) + \int_\G \un_{\|g\|>1/\varepsilon} \|g^{-1} \|^\delta \|x\|^{-\delta} \di\rho^{\ast n}(g) \\
& \leqslant \|x\|^{-\delta} \left( \int_\G e^{-\delta \ln \frac{\|gx\|}{\|x\|}} \di\rho^{\ast n}(g) + \int_\G \un_{\|g\|>1/\varepsilon} \|g^{-1}\|^\delta \di\rho^{\ast n}(g)\right)
\end{flalign*}
Moreover, there is $\delta_0 \in \R_+^\ast$ such that for any $\delta \in ]0,\delta_0]$ there are $C, t \in \R_+^\ast$ such that for any $n\in \N$,
\[
\sup_{x\in \R^d\setminus\{0\}} \int_\G e^{-\delta \ln\frac{\|gx\|}{\|x\|}} \di\rho^{\ast n}(g) \leqslant Ce^{-tn}
\]
(we refer to~\cite{BL85} theorem~6.1, for a proof of this result).

And so, we get that for any $x\in B(0,\varepsilon) \setminus\{0\}$,
\[
P^n u_\delta (\overline x) \leqslant u_\delta(\overline x) \left( Ce^{-tn} + \int_\G \un_{\|g\|>1/\varepsilon} \|g^{-1}\|^\delta \di\rho^{\ast n}(g) \right)
\]
Let $n_0$ be such that $Ce^{-tn_0} \leqslant 1/4$ and $\varepsilon$ such that
\[
\int_\G \un_{\|g\|>1/\varepsilon} \|g^{-1}\|^\delta \di\rho^{\ast n_0}(g)\leqslant 1/4
\]
(such an $\varepsilon $ exists since $\rho$ has an exponential moment).

What we get is that for this choice of $n_0$ and $\varepsilon$, for any $\overline x\in B(0,\varepsilon) \setminus\{0\}$,
\[
P^{n_0} u_\delta (x) \leqslant \frac 1 2 u_\delta(x)
\]
Moreover, if $\overline x$ is on the complement set of the ball,
\begin{flalign*}
P^n u_\delta(\overline x) &= \int_\G d(g\overline x,0)^{-\delta} \di\rho^{\ast n}(g) \leqslant d(\overline x,0)^{-\delta} \int_\G \|g^{-1}\|^{\delta} \di\rho^{\ast n}(g) \\ & \leqslant \varepsilon^{-\delta} \int_\G \|g^{-1}\|^{\delta} \di\rho^{\ast n}(g)
\end{flalign*}
and this finishes the proof of the lemma.
\end{proof}

From now on, we fix $\delta \in \R_+^\ast$ such that the function $u_\delta$ satisfies $P^{n_0} u_\delta \leqslant a u_\delta + b$ for some $n_0 \in \N^\ast$, $a\in [0,1[$ and $b\in \R$. Let $a_1\in ]a,1[$ be such that $a_1^{-n_0} a \leqslant 1$.

Note
\[
u_0 = \sum_{k=0}^{n_0-1} a_1^{-k} P^k u_\delta
\]
Then,
\begin{flalign*}
P u_0 &= \sum_{k=0}^{n_0-1} a_1^{-k} P^{k+1} u_\delta = a_1 \sum_{k=1}^{n_0-1} a_1^{-k} P^k u_\delta + a^{-(n_0-1)} P^{n_0} u_\delta \\
& \leqslant a_1 \sum_{k=1}^{n_0-1} a_1^{-k} P^k u_\delta +  a_1^{-(n_0-1)} (a u_\delta +b) \\
& \leqslant a_1 u_0(x) + ba_1^{-(n_0 -1)}
\end{flalign*}
Moreover, as
\[
u_\delta(x) \int_\G \|g\|^{-\delta} \di\rho^{\ast k}(g) \leqslant P^k u_\delta(x) = \int_\G \|gx\|^{-\delta} \di\rho^{\ast k}(g) \leqslant  u_\delta(x) \int_\G \|g^{-1}\|^\delta \di\rho^{\ast k}(g),
\]
the function $u_0$ that we constructed is also equivalent to $d(x,0)^{-\delta}$ or more specifically,
\[
0< \inf_{x\in \T^d\setminus 0} \frac {u_0(x)}{ d(x,0)^{-\delta}} < \sup_{x\in \T^d\setminus\{0\} } \frac {u_0(x)}{ d(x,0)^{-\delta}} < +\infty
\]
So what we just proved is the following
\begin{proposition} \label{proposition:recurrence_loin_0}
Let $\rho$ be a borelian probability measure on $\mathrm{SL}_d(\Z)$ whose support generates a strongly irreducible and proximal group and which has an exponential moment.

Then, there is $\delta\in \R_+^\ast$, $a\in[0,1[$, $b\in \R$ and a function $u_0$ on $\T^d$, such that
\[
0< \inf_{x\in \T^d\setminus 0} \frac {u_0(x)}{ d(x,0)^{-\delta}} < \sup_{x\in \T^d\setminus\{0\} } \frac {u_0(x)}{ d(x,0)^{-\delta}} < +\infty
\]
and
\[
P u_{0}\leqslant au_0 +b
\]
\end{proposition}

Now, we are going to use this function $u_0$ to construct some other that will allow us to prove that if $x$ is not well approximable by rational points, then, in $\rho-$average, so are the $gx$.

What we will do is, for a fixed diophantine condition  $\varphi$, constructing $u_\varphi$ such that $Pu_{\varphi} \leqslant au_\varphi + b$ and $u_\varphi$ is finite on points satisfying the condition $\varphi$.

\medskip
For $Q \in \N^\ast$, we note $\X_Q$ the set of primitives elements in $\frac 1 Q \Z^d/ \Z^d$ that is to say, the set of elements of $\frac 1 Q \Z^d/\Z^d$ that doesn't belong to $\frac 1 {q} \Z^d/\Z^d$ for $q <Q$. 

Then, $\X_Q$ is $\mathrm{SL}_d(\Z)-$invariant : indeed, if $p \in \X_Q$ then $gp \in \frac 1 Q \Z^d / \Z^d$ since $g$ has integer coefficients and $gp$ can not belong to $\frac 1 q \Z^d / \Z^d$ with $q<Q$ because if it was so, so would $p= g^{-1} g p$.

\medskip
Let $\varphi : \N \to \R_+^\ast$ be a strictly non decreasing function. For $x\in \T^d \setminus\{0\}$, we note
\[
u_\varphi(x) = \sum_{ Q\in \N^\ast} \frac 1{ \varphi(Q)^\delta } \sum_{p\in \X_Q} u_0 (x-p)
\]
This function $u_\varphi$ is proper (it is non negative and lower semi-continuous)

Moreover, it carries the diophantine properties of $x$.

Indeed, by definition of $h_\varphi(x)$ (see the previous section), we have that for some constant $C$ that doesn't depend on $\varphi$,
\[
h_\varphi(x)^\delta \leqslant C u_\varphi(x)
\]
and reciprocally, if $\varphi': \R\to \R_+^\ast$ is an other strictly non decreasing function such that $\varphi'(Q) \in \cal O(\varphi(Q) Q^{-(d+2)/\delta)})$ then,
\[
u_\varphi(x) \leqslant C h_{\varphi'}(x)^\delta \sum_{Q\in \N^\ast} Q^d \left(\frac{ \varphi'(Q)}{\varphi(Q)} \right)^{\delta}
\]
and so, if $h_{\varphi'}(x)$ is finite, so is $u_\varphi(x)$.

\medskip
Thus, controlling $u_\varphi(x)$ is controlling the diophantine properties of $x$ and reciprocally.

The aim of this construction is the following
\begin{lemma} \label{lemma:definition_u_varphi}
Let $u_0$ be the function constructed in the previous lemma.

Let $\varphi: \N\to \R_+^\ast$ be a strictly non decreasing function such that
\[
\sum_n \frac{n^d}{\varphi(n)^\delta} <+\infty
\]

For $x\in \T^d$, note
\[
u_\varphi(x) = 1+ \sum_{ Q\in \N^\ast} \frac 1{\varphi(Q) ^\delta} \sum_{p\in \X_Q} u_0 (x-p) 
\]

Then, there are $a\in ]0,1[$ and $b\in \R$ such that
\[
Pu_\varphi\leqslant au_\varphi+b
\]
\end{lemma}

\begin{remark}
One has to think of $\varphi$ has growing very fast (we will take $\varphi(n) = e^{Bn^\beta}$) so the summability assumption will always be satisfied and multiplying $\varphi$ by a polynomial function doesn't really change the points where $u_\varphi$ takes finite values. Therefore, it is almost the same thing to say that $u_\varphi(x)$ is finite or that $h_\varphi(x)$ is.
\end{remark}

\begin{proof}
Let's remind that $Pu_0 \leqslant a u_0 + b$ for some $a\in ]0,1]$ and $b\in \R$.

And so, if we note, for $Q \in \N^\ast$ and $x\in \T^d \setminus \Q^d/\Z^d$,
\[
u_Q(x) = \sum_{p\in \X_Q} u_0(x-p)
\]
we have, using that $\mathrm{SL}_d(\Z)$ permutes $\X_Q$, that
\begin{flalign*}
Pu_Q(x) &= \int_\G \sum_{p\in \X_Q} u_0(gx-p )\di\rho(g) = \int_\G \sum_{p \in \X_Q} u_0(g(x-p)) \di\rho(g) \\
& = \sum_{p\in \X_Q} Pu_0(x-p) \leqslant a\sum_{p\in \X_Q} u_0(x-p) + b|\X_Q| \\
& \leqslant au_Q(x) + bQ^d
\end{flalign*}
where we used that $|\X_Q| \leqslant Q^d$.

And so,
\[
P(u_\varphi) (x) \leqslant 1 +\sum_{Q\in \N^\ast} \frac 1 { \varphi (Q)^\delta} Pu_Q(x) \leqslant au_\varphi(x) + 1-a + b\sum_{Q\in \N^\ast} \frac {Q^d} { \varphi(Q)^\delta} 
\]
\end{proof}

We are finally able to solve Poisson's equation for hölder-continuous functions in next
\begin{corollary} \label{corollaire:solution_poisson_tore_diophantien}
Under the hypothesis of proposition~\ref{proposition:vitesse_convergence_tore}, for any $\gamma \in ]0,1]$ and any $M \in \R_+^\ast$, there is $\beta_0\in \R_+^\ast$ such that for any $B\in \R_+^\ast$ and any $\beta \in ]0,\beta_0[$, there is a constant $C$ such that, noting $\varphi(n)= e^{B n^\beta}$, we have that for any $x$ such that
\[
u_\varphi(x) <+\infty
\]
we have that
\[
\cal W_\gamma (\rho^{\ast n} \ast \delta_x,\nu) \leqslant \frac C{n^{1+M}} u_\varphi(x)
\]
In particular, for any $\gamma-$hölder continuous function $f$ on the torus, there is $g\in \cal F^3_{u_\varphi}$ (cf section~\ref{section:drift}) such that,
\[
f=g-Pg + \int f\di\nu \text{ on }\{u_\varphi<+\infty\} \text{ and }\|g\|_{\cal F_u^3} \leqslant C \|f\|_\gamma
\]
\end{corollary}

\begin{proof}
We apply proposition~\ref{proposition:vitesse_convergence_tore} noting that in this case, there is a consant $C$ such that for any $n\in \N^\ast$,
\[
\cal W_\gamma (\rho^{\ast n } \ast \delta_x,\nu) \leqslant \frac{C}{n^{1+M}}h_\varphi(x)^{\delta/3} \leqslant \frac{C}{n^{1+M}}u_\varphi(x)
\]
and so,
\[
\left(\sum_n \cal W_\gamma(\rho^{\ast n } \ast \delta_x,\nu)\right)^3 \leqslant C^3 u_\varphi(x) \left(1+\sum_{n\in \N^\ast} \frac 1 {n^{1+M}} \right)^3
\]
So, we can set
\[
g = \sum_{n\in \N} P^n \left(f-\int f\di\nu\right)
\]
noting that, by definition of Wasserstein's distance, for any $n\in \N$ and any $x\in \X$,
\[
\left|P^nf(x) - \int f\di\nu\right|\leqslant \|f\|_\gamma \cal W_\gamma(\rho^{\ast n } \ast \delta_x,\nu)
\]
\end{proof}

\subsection{Central limit theorem and law of the iterated logarithm}

\begin{miniabstract}
In this section, we use the result of the previous ones to finally prove the central limit theorem and the law of the iterated logarithm for the random walk on the torus.\end{miniabstract}

As we now know with corollary~\ref{proposition:vitesse_convergence_tore}, holder continuous functions $f$ on the torus writes $f=g-Pg + \int f \di\nu$ where $g$ is dominated by a drift function finite on points badly approximalble by rationals. We are going to prove the the validity of ``law of large numbers''-type hypothesis in corollary~\ref{corollary:TCL_martingales} and this will allow us to prove the central limit theorem and the law of the iterated logarithm. We don't know how to prove the law of large numbers for functions of $\eupp$ and this is why we will go back to the function $f$ to use the speed of convergence given by our corollary or Bourgain-Furmann-Lindenstrauss-Mozes's theorem.

\medskip
We will need the following
\begin{lemma} \label{lemme:preliminaire_variance}
Let $\rho$ be a borelian probability measure on $\mathrm{SL}_d(\Z)$ whose support generates a strongly irreducible and proximal group and which has an exponential moment.

For any $\gamma \in ]0,1]$, there is $\alpha_0 \in \R$ such that for any $\alpha \in ]\alpha_0, +\infty[$ there is $\beta_0 \in \R_+^\ast$ such that for any $\beta \in ]0,\beta_0[$ and any $B\in \R$, noting $\varphi(q) = e^{Bq^\beta}$, we have that for any sequence $(f_n)$ of $\gamma-$hölder-continuous functions on the torus and such that $\int f_n \di\nu=0$,
\[
\esp_x \left| \sum_{k=0}^{n-1} f_k(X_k) \right|^4 = \cal O \left(\frac{n^3}{(\ln n)^{\alpha-1}} u_\varphi(x) \max_{k\in \lib0,n\rib} \|f_k\|_\gamma^4 \right)
\]
Where the involved constant doesn't depend on $n$, $x$ nor the sequence $(f_n)$.
\end{lemma}

\begin{remark}
What is hidden behind this lemma is a kind of Burckholder inequality that says that if $(Y_i)$ is a sequence of iid bounded random variables on $\R$ of null expectation, then for any $r \in \N$,
\[
\esp\left| \sum_{k=0}^{n-1} Y_k\right|^{r} \in \cal O \left( n^{r/2} \right)
\]
\end{remark}

\begin{proof}
First, we choose $\alpha\in \R_+^\ast$ and we will see a lower bound on $\alpha$ later. We note $\psi(n) = n^{-\alpha}$ and according to proposition~\ref{proposition:vitesse_convergence_tore}, there is a constant $C$ such that for any $n\in \N$ and any $x\in \X$,
\[
\cal W_\gamma(\rho^{\ast n} \ast \delta_x, \nu) \leqslant \frac{C}{n^{\alpha}} u_\varphi(x)
\]
where $\varphi(q) = e^{B q^\beta}$.
Moreover, for $n\in \N$, we note
\[
S_n = \sum_{k=0}^{n-1} f_k(X_k)
\]
We can compute
\begin{flalign*}
\esp_x |S_{n+1}|^4 &= \sum_{k=0}^{4} \binom 4 k \esp_x f_{n}(X_{n})^k S_{n}^{4-k} \\
&=\esp_x |S_{n}|^4 + 4 \esp_x f_{n}(X_{n}) S_{n}^3 + \sum_{k=2}^4 \binom 4 k \esp_x (f_{n}(X_{n}))^k S_{n}^{4-k}
\end{flalign*}
So, we note
\[
A_n := \esp_x f_{n}(X_{n}) S_{n}^3  \text{ et }B_n:=\sum_{k=2}^4 \binom 4 k \esp_x (f_{n}(X_{n}))^k S_{n}^{4-k}
\]
and so, we have, noting $p(n)$ (and even only $p$ to simplify notations) a sequence that we will determine later and such that $0 \leqslant p(n) \leqslant n$, that
\begin{flalign*}
A_n &= \sum_{k=0}^3 \binom 3 k \esp_x f_{n}(X_{n}) (S_{n} - S_{p})^k S_{p}^{3-k} \\
&=\esp_x f_{n}(X_{n}) S_{p}^3 +3 \esp_x f_n(X_n) (S_n- S_p)S_p^2 \\ & \retrait\retrait +\sum_{k=2}^3 \binom 3 k \esp_x f_{n}(X_{n}) (S_{n} - S_{p})^k S_{p}^{3-k} 
\end{flalign*}
We note each of this terms $A_n^1, A_n^2$ et $ A_n^3$.

Then, using the fact that $\int f_n \di\nu=0$, and that, according to proposition~\ref{proposition:vitesse_convergence_tore},
\[
|P^{n-p+2} f_n (X_{p-1})| \leqslant \frac C {(n-p+2)^\alpha} \|f_n\|_\gamma u(X_{p-1})
\]
and that
\[
P^l u(x) \leqslant a^l u(x) + \frac b{1-a}
\]
we get that
\begin{flalign*}
|A_n^1| &= \left|\esp_x P^{n-p+2} f_{n} (X_{p-1}) S_{p}^3 \right| \\
& \leqslant \frac C {(n-p+2)^{\alpha}} \|f_{n}\|_\gamma \esp_x u(X_{p-1}) |S_{p}|^3 \\
& \leqslant \frac C {(n-p+2)^{\alpha}} \|f_{n}\|_\gamma \|S_p\|_\infty^3 P^{p-1} u(x) \\
& \leqslant \frac C {(n-p+2)^{\alpha}} \|f_{n}\|_\gamma \|S_p\|_\infty^3 \left(a^{p-1} + \frac b{1-a}\right)u(x) \\
&= \cal O \left(u(x) \max_{k\in \lib 0,n\rib} \|f_k\|_\gamma^4 \frac{n^3}{(n-p+2)^{\alpha}} \right)
\end{flalign*}
Moreover,
\begin{flalign*}
A_n^2 &= \sum_{k=p}^{n-1} \esp_x f_n(X_n) f_k(X_k) S_p^2 \\
&= \sum_{k=p}^{q-1} \esp_x  P^{n-k} f_n(X_k) f_k(X_k) S_p^2 + \sum_{k=q}^{n-1} \esp_x P^{k-p} (f_k P^{n-k} f_n)(X_p) S_p^2
\end{flalign*} 
and so, for some sequence $q(n)$ that we will determine later and with $p(n) < q(n)<n$, we have that
\begin{flalign*}
|A_n^2| & \leqslant \sum_{k=p}^{q-1} \frac{C}{(n-k)^{\alpha}} \|f_k\|_\infty  \|f_n\|_\gamma\esp_x u(X_k) S_p^2 \\ & \retrait +\sum_{k=q}^{n-1} \frac C{(k-p)^{\alpha}} \|f_k P^{n-k} f_n\|_\gamma \esp_x u(X_p) S_p^2 + \left|\int f_k P^{n-k} f_n \di\nu\right| \esp_x S_p^2 \\
&  \leqslant n^2 \max_{k\in \lib 0,n\rib} \|f_k \|_\gamma^4 \left( \sum_{k=p}^{q-1} \frac {CP^k u (x)} {(n-k)^{\alpha}}  + \sum_{k=q}^{n-1} \frac{ C\|P\|_\gamma^{n-k} P^p u(x) }{(k-p)^{\alpha}}  +  \|P\|_{\mathrm{L}^2_0(\X,\nu)}^q \right) \\
& = \cal O \left( n^2 \max_{k\in \lib 0,n\rib} \|f_k \|_\gamma^4 u(x) \left( \sum_{k=n-q+1}^{n-p} \frac 1 {k^{\alpha}} +  \frac{\|P\|_\gamma^{n-q}}{(q-p)^\alpha} +\|P\|_{\mathrm{L}^2_0(\X,\nu)}^q \right) \right)
\end{flalign*}
So, with $q = n-\ln n$ and $p= n - n^\delta$ with $\delta <1/2$, and taking $\alpha_0$ such that $\delta \alpha_0 > \ln \|P\|_\gamma$, we find that
\begin{flalign*}
A_n^2 &= \cal O \left(n^2 \max_{k\in \lib 0,n\rib} \|f_k \|_\gamma^4 u(x)\left( \sum_{k=\ln n+1}^{+\infty} \frac 1 {k^\alpha} + \frac{n^{\ln \|P\|_\gamma}}{(n^\delta - \ln n)^\alpha} + \|P\|^{n-\ln n}_{\mathrm{L}^2_0(\X,\nu)}\right) \right) \\
& = \cal O \left( \frac{n^2}{ (\ln n)^{\alpha-1}} \max_{k\in \lib 0,n\rib} \|f_k\|_\gamma^4 u(x)\right)
\end{flalign*}
and finally,
\begin{flalign*}
|A_n^3| & \leqslant \sum_{k=2}^3 \binom 3 k \|f_n\|_\infty \|S_n - S_p\|^k_\infty \esp_x |S_{p}|^{3-k}
\end{flalign*}
But,
\[
\|S_n - S_p\|_\infty \leqslant \sum_{k=p}^{n-1} \|f_p\|_\infty \leqslant (n-p) \sup_{k\in \lib 0, n-1\rib} \|f_k\|_\infty
\]
so
\[
A_n^3 = \cal O \left(n^{1+2\delta} \max_{k\in \lib0,n\rib} \|f_k\|_\gamma^4 \right)
\]
and we recall that we choose $\delta<1/2$.
So we can take $\delta = 1/4$ and we can assume that $\delta \alpha_0 >1$ to get that
\[
|A_n^1| = \cal O \left(u(x) \max_{k\in \lib 0,n\rib} \|f_k\|_\gamma^4 n^{3-\delta \alpha} \right)
\]
thus we proved that
\[
A_n = \cal O \left( \frac{n^2}{(\ln n)^{\alpha-1}}  u(x)\max_{k\in \lib 0,n\rib} \|f_k\|^4 \right)
\]

To study $B_n$, remark in a first time that
\[
\esp_x S_n^2 = \sum_{k=0}^{n-1} P^k (f_k)^2(x) + 2 \sum_{k=0}^{n-1} \sum_{l=0}^{k-1}P^l (f_l P^{k-l} f_k)(x)
\]
The first term of this sum is dominated by $n\max_{k\in \lib 0,n\rib} \|f_k\|_\gamma$ and a computation similar to the previous one proves that the second one is bounded by constants times
\[
\frac{n^2}{(\ln n)^{\alpha-1}} 
u(x)\max_{k\in \lib 0,n\rib} \|f_k\|^2 .
\]
Therefore,
\[
\esp_x (f_n(X_n))^2 S_n^2 \in \cal O \left(\frac{n^2}{(\ln n)^{\alpha-1}} 
u(x)\max_{k\in \lib 0,n\rib} \|f_k\|^4 \right)
\]
And this proves that
\[
B_n = \cal O\left(\frac{n^2}{(\ln n)^{\alpha-1}} 
u(x)\max_{k\in \lib 0,n\rib} \|f_k\|^4 \right)
\]
So,
\[
\esp_x |S_{n+1}|^4 = \esp_x |S_n|^4 +\cal O \left(\frac{n^2}{(\ln n)^{\alpha-1}} 
u(x)\max_{k\in \lib 0,n\rib} \|f_k\|^4 \right)
\]
and iterating this relation, we get that
\[
\esp_x |S_{n+1}|^4 = \cal O \left(\frac{n^3}{(\ln n)^{\alpha-1}}u(x)\max_{k\in \lib 0,n\rib} \|f_k\|^4  \right)
\]
which is what we intended to prove.
\end{proof}

We are now ready to prove the convergence of the variance in next
\begin{lemma} \label{lemma:convergence_variance_tore_diophantien}
Let $\rho$ be a borelian probability measure on $\mathrm{SL}_d(\Z)$ whose support generates a strongly irreducible and proximal group and which has an exponential moment.

Then, for any $\gamma \in ]0,1]$ there is $\beta_0 \in \R_+^\ast$ such that for any $\beta \in ]0,\beta_0[ $ and any $B\in \R$, noting $\varphi(q) = e^{Bq^\beta}$, we have that for any $\gamma-$hölder-continuous function $f$ on the torus, noting $g$ the solution to Poisson's equation defined in $\mathcal F^3_{u_\varphi}$ and given by corollary~\ref{corollaire:solution_poisson_tore_diophantien} we have that for any $x\in \X$ such that $u_\varphi(x)$ is finite,
\[
\frac 1 n \sum_{k=0}^{n-1} P(g^2)(X_k) - (Pg(X_k))^2 \xrightarrow\, \int g^2 - (Pg)^2 \di \nu \;\; \prob_x-\text{a.e. and in }\mathrm{L}^1(\prob_x)
\]
\end{lemma}

\begin{proof}
We take at first $\alpha_0$ equal to the one of the previous lemma and take $\alpha>\alpha_0$. It comes with it a constant $\beta_0$ such that for any $B\in \R$ and any $\beta \in ]0,\beta_0[$, the function $\psi(t)$ given by~\ref{proposition:vitesse_convergence_tore} satisfies that $\sup_n n^{\alpha} \psi(n) $ is finite.

Remark that for any $\gamma-$hölder-continuous function $f$ on the torus, the function $g$ given by proposition~\ref{proposition:vitesse_convergence_tore} is square-integrable agains Lebesgue's measure. We can see this as a consequence of~\ref{lemma:extmesures} or, more simply, use that under our assumptions, the operator $P$ has a spectral gap $\mathrm{L}^2(\X,\nu)$ as we already saw in the introduction and so, the function $g$ is $a.e.-$equal to a square-integrable function. We will actually use this spectral gap in the proof of this lemma.

We assume without any loss of generality that $\int f\di\nu=0$.
To prove the lemma, we use that $f=g-Pg$ to write
\begin{flalign*}
I_n(x) :&=\frac 1 n \sum_{k=0}^{n-1} P(g^2)(X_k) - (Pg(X_k))^2 = \frac 1 n \sum_{k=0}^{n-1} P(g^2)(X_k) - (g(X_k) - f(X_k))^2 \\
&= \frac 1 n \sum_{k=0}^{n-1} P(g^2)(X_k) - g^2(X_k) - \frac 1 n \sum_{k=0}^{n-1} (f(X_k))^2 + \frac 2 n \sum_{k=0}^{n-1} f(X_k) g(X_k)
\end{flalign*}
According to proposition~\ref{proposition:lln_martingales}, and using that $u$ is a drift function and that $g \in \mathcal F^3_u$, we get that for any $x$ such that $u(x)$ is finite,
\[
\frac 1 n \sum_{k=0}^{n-1} P(g^2)(X_k) - g^2(X_k) \xrightarrow\, 0 \text{ in }\mathrm{L}^1(\prob_x) \text{ and }\prob_x-\text{a.e.}
\]
Moreover, the law of large numbers proves that for any irrational point $x$ of the torus (and so in particular, for any $x$ such that $u(x)$ is finite),
\[
\frac 1 n \sum_{k=0}^{n-1}(f(X_k))^2 \xrightarrow\, \int_\X f^2 \di\nu \text{ in }\mathrm{L}^1(\prob_x) \text{ and }\prob_x-\text{a.e.}
\]
Moreover, if $p : \N \to \N$ is a non decreasing function converging to infinity that we will determine later, we have that
\[
g(x) = \sum_{l=0}^{p(k)-1} P^l f(x) + \sum_{l=p(k)}^{+\infty} P^l f(x)
\]
and so
\[
\frac 1 n \sum_{k=0}^{n-1} f(X_k) g(X_k) =  \frac 1 n \sum_{k=0}^{n-1} \sum_{l=0}^{p(k)-1} f(X_k) P^l f (X_k) + \frac 1 n \sum_{k=0}^{n-1} \sum_{l=p(k)} ^{+\infty} f(X_k) P^l f(X_k)
\]
But, according to lemma~\ref{lemme:preliminaire_variance} applied to the sequence of functions
\[
f_n = \sum_{k=0}^{p(n)-1} fP^l f- \int fP^l f\di \nu
\]
we have that
\[
\esp_x\left|\sum_{k=0}^{n-1} \sum_{l=0}^{p(k)-1} \left(f(X_k) P^l f(X_k) - \int f P^l f \di\nu \right) \right|^4 = \cal O \left( \frac {n^3 u(x)}{(\ln(n))^{\alpha-1}} \max_{k\in \lib 0,n\rib} \left\| f_k\right\|_\gamma^4\right)
\]
and, for any $k\in \lib0,n\rib$,
\[
\left\| f \sum_{l=0}^{p(k)-1} P^lf \right\|_\gamma \leqslant \|f\|_\gamma^2\sum_{l=0}^{p(k)-1} \|P\|_\gamma^l  \leqslant \|f\|_\gamma^2 \frac{ \|P\|_\gamma^{p(k)} }{\|P\|_\gamma  - 1}
\]
Thus,
\[
\esp_x\left|\sum_{k=0}^{n-1} \sum_{l=0}^{p(k)-1} \left(f(X_k) P^l f(X_k) - \int f P^l f \di\nu \right) \right|^4 = \cal O \left( \frac {n^3}{(\ln(n))^{\alpha_0}} u(x) \|f\|_\gamma^8 \|P\|_\gamma^{p(n)} \right)
\]
so, if $p(n) \asymp \delta_1 \ln (\ln n)$ with $\delta_1$ such that $\delta_1 \ln \|P\|_\gamma<\alpha_0$, we have that for $n$ large enough,
\[
\|P\|_\gamma^{p(n)} \leqslant e^{\delta_1 \ln ( \ln n) \|P\|_\gamma} = (\ln n)^{\delta_1\ln \|P\|_\gamma}
\]
and so,
\[
\sum_n \frac 1 {n^4} \esp_x\left|\sum_{k=0}^{n-1} \sum_{l=0}^{p(k)-1} \left(f(X_k) P^l f(X_k) - \int f P^l f \di\nu \right) \right|^4 <+\infty
\]
This proves that
\[
\frac 1 n \sum_{k=0}^{n-1} \sum_{l=0}^{p(k)-1} f(X_k) P^l f(X_k) - \int fP^l f \di\nu \xrightarrow\, 0\;\; \prob_x\text{-a.e. and in }\mathrm{L}^1(\prob_x)
\]
Moreover, using the spectral gap in $\mathrm{L}^2(\X,\nu)$ and Cesaro's lemma, we get that
\[
\frac 1 n\sum_{k=0}^{n-1} \sum_{l=0}^{p(k)-1} \int fP^l f\di\nu \xrightarrow\, \sum_{l=0}^{+\infty }\int f P^l f\di\nu
\]
We are going to prove that $\frac 1 n \sum_{k=0}^{n-1} \sum_{l=p(k)} ^{+\infty} f(X_k) P^l f(X_k)$ converges to $0$. But, using proposition~\ref{proposition:vitesse_convergence_tore}, we have that
\begin{flalign*}
\frac 1 n \left|\sum_{k=0}^{n-1} \sum_{l=p(k)} ^{+\infty} f(X_k) P^l f(X_k) \right| &\leqslant \frac 1 n \sum_{k=0}^{n-1} \sum_{l=p(k)}^{+\infty} |f(X_k)| |P^l f(X_k) |  \\
&\leqslant \frac 1 n \sum_{k=0}^{n-1} \sum_{l= p(k)}^{+\infty} \| f\|_\infty \frac{ C }{ l^{1 + \alpha
}} h_\varphi(X_k)^{\delta/3} \| f\|_\gamma \\
& \leqslant \frac 1 n \sum_{k=0}^{n-1} \frac{C'}{p(k)^{\alpha}} h_\varphi(X_k)^{\delta/3} \| f\|_\gamma^2
\end{flalign*}
for some constant $C'$.

But, by definition of $u_\varphi$, $h_\varphi^{\delta/3} \in \mathcal F_u^3$ and so, according to lemma~\ref{lemma:produit_cesaro},  we have that
\[
 \frac 1 n \sum_{k=0}^{n-1} \frac{1}{p_k^{\alpha}} h_\varphi(X_k)^{\delta/3} \xrightarrow\, 0 \;\;\prob_x-\text{a.e. and in }\mathrm{L}^1(\prob_x)
\]
What we just proved is that
\[
\frac 1 n \sum_{k=0}^{n-1} P(g^2)(X_k) - (Pg(X_k))^2 \xrightarrow\, -\int f^2\di\nu + 2 \sum_{l=0}^{+\infty} fP^l f\di\nu \;\;\prob_x\text{-a.e. and in }\mathrm{L}^1(\prob_x)
\]
To conclude, we only have to remark that
\begin{flalign*}
-\int f^2\di\nu + 2 \sum_{l=0}^{+\infty} fP^l f\di\nu &= \int - (g-Pg)^2 + 2(g-Pg)g \di \nu \\
&= \int 2g^2 -2g Pg - g^2 +2gPg - (Pg)^2\di\nu \\&= \int g^2 - (Pg)^2 \di\nu
\end{flalign*}
which finishes the proof of the lemma.
\end{proof}

\begin{theorem}\label{theorem:TCL_tore_diophantien}
Let $\rho$ be a borelian probability measure on $\mathrm{SL}_d(\Z)$ whose support generates a strongly irreducible and proximal group and which has an exponential moment.

Then, for any $\gamma \in ]0,1]$ there is $\beta_0 \in \R_+^\ast$ such that for any $B \in \R_+^\ast$ and $\beta \in ]0,\beta_0[$ we have that for any irrational point $x\in \X$ such that the inequality
\[
d\left(x,\frac p q \right) \leqslant e^{-Bq^\beta}
\]
has a finite number of solutions $p/q \in \Q^d/\Z^d $, we have that for any $\gamma-$holder continuous function $f$ on the torus, noting $\sigma^2(f) $ the variance given by equation~\ref{equation:variance} we have that
\[
\frac 1 {\sqrt n} \sum_{k=0}^{n-1} f(X_k) \xrightarrow{\mathcal{L}} \cal N\left(\int f\di \nu, \sigma^2(f) \right)
\]
(If $\sigma^2=0$, the law $\cal N(\mu,\sigma^2)$ is a Dirac mass at $\mu$).

Moreover, if $\sigma^2(f) \not=0$ then, $\prob_x-$a.e.,
\[
\liminf \frac{ \sum_{k=0}^{n-1} f(X_k) - \int f \di \nu}{\sqrt{2n\sigma^2(f) \ln\ln n}} =-1 \text{ and }\limsup \frac{ \sum_{k=0}^{n-1} f(X_k) - \int f \di \nu}{\sqrt{2n\sigma^2(f) \ln\ln n}} =1 
\]
and if $\sigma^2(f)=0$, then for $\nu-$a.e. $x\in \X$, the sequence $ (\sum_{k=0}^{n-1} f(X_k) - \int f \di \nu)_n$ is bounded in $\mathrm{L}^2(\prob_x)$.
\end{theorem}

\begin{proof}
This is a direct corollary of lemma~\ref{lemma:convergence_variance_tore_diophantien}, lemma~\ref{lemma:u(X_n)/n} and proposition~\ref{corollary:TCL_martingales}.

The condition on $\sigma^2(f)$ comes from corollary~\ref{corollary:nullite_variance} if we note that since $Pu_\varphi \leqslant au_\varphi+b$, we have, for any $n\in \N$,
\[
P^n u_\varphi \leqslant a^n u_\varphi + \frac{b}{1-a}
\]
and so, for any $x$ satisfying the diophantine condition, $\sup_n P^n (g^2)(x) \leqslant u_\varphi(x) + \frac b {1-a}$ is finite. And moreover, $\nu(x| u_\varphi(x)<+\infty)=1$.
\end{proof}

\bibliographystyle{amsalpha}
\bibliography{biblio}

\end{document}